\documentclass{article}
\usepackage[utf8]{inputenc}
\usepackage{bbm,amssymb,amsmath,dsfont,bbm,amsthm,geometry,enumitem}
\usepackage[backend=bibtex,style=alphabetic,sorting=nyt,doi=false,giveninits=true]{biblatex}
\usepackage[affil-it]{authblk} 
\AtEveryBibitem{\clearfield{month}}
\AtEveryBibitem{\clearfield{ISSN}}
\AtEveryBibitem{\clearfield{number}}
\AtEveryCitekey{\clearfield{month}}
\renewbibmacro{in:}{}
\DeclareFieldFormat{pages}{#1}
\usepackage[english]{babel}
\theoremstyle{plain}
\newtheorem{theorem}{Theorem}[section]
\theoremstyle{definition}
\newtheorem{definition}[theorem]{Definition}
\theoremstyle{definition}
\newtheorem{notation}[theorem]{Notation}
\theoremstyle{plain}

\theoremstyle{plain}

\theoremstyle{plain}
\newtheorem{lemma}[theorem]{Lemma}
\theoremstyle{definition}
\newtheorem{exmp}[theorem]{Example}
\theoremstyle{definition}

\usepackage[pagewise]{lineno}
\usepackage{theoremref}
\usepackage{hyperref}

\theoremstyle{definition}
\newtheorem{remark}[theorem]{Remark}
\usepackage[nottoc]{tocbibind}
\providecommand{\keywords}[1]
{
  \small	
  \textbf{\textit{Keywords---}} #1
}
\begin{document}
\title{{Closest Distance between Iterates of Typical Points}}
\author{Boyuan Zhao\\ \footnotesize{Email: \href{mailto:bz29@st-andrews.ac.uk}{bz29@st-andrews.ac.uk} }\footnote{BZ is supported by the China Scholarship Council for his PhD programme. }}
\affil{\footnotesize Department of Mathematics, University of St Andrews}
\date{}
\maketitle
\begin{abstract}
\noindent
    The shortest distance between the first $n$ iterates of a typical point can be quantified with a log rule for some dynamical systems admitting Gibbs measures. We show this in two settings. For topologically mixing Markov shifts with at most countably infinite alphabet admitting a Gibbs measure with respect to a locally H\"{o}lder potential, we prove the asymptotic length of the longest common substring for a typical point converges and the limit depends on the R\'{e}nyi entropy. For interval maps with a Gibbs-Markov structure, we prove a similar rule relating the correlation dimension of Gibbs measures with the shortest distance between two iterates in the orbit generated by a typical point.
\end{abstract}
\section{Introduction}
Consider a topological Markov shift $(\Sigma_A,\sigma,\mathcal{I})$ on a (at most) countably infinite alphabet $\mathcal{I}$ with respect to a transition matrix $A$ equipped with the natural symbolic metric $d$. The first $n$ iterates of a $\underline{x}\in\Sigma_A$ under $\sigma$ are the initial $n$ symbols appearing in $\underline{x}$. We are interested in the asymptotic behaviour of the following quantity:
\begin{align*}
M_n(\underline{x})
=\max\{k:\text{$\exists$ $0\leq i<j\leq n-1$: }x_i,\dots,x_{i+k-1}=x_j,\dots,x_{j+k-1}\},
\tag{1.1}\label{ineqn:1.1}
\end{align*}
which counts the maximum length of self-repetition. Studying quantities of this type is often referred to as the longest common substring matching problem. One motivation comes from the matching of nucleotide sequences in DNA, and early results were established in the 80s by Arratia and Waterman's work \cite{ARRATIA198513}. They showed that the length of the longest common substring among two i.i.d stochastic sequences $X_1,X_2,\dots$ and $Y_1,Y_2,\dots$ taking letters in a finite alphabet with uniform distribution, \[M_n(X,Y):=\sup\{k:X_{i+m}=Y_{j+m}\text{ for all }m=1 \text{ to }k\text{ and }1\leq i,j\leq n-k\}\]satisfies an Erd\H{o}s-R\'{e}nyi law
\[\mathbbm{P}\left(\lim_{n\rightarrow\infty}\frac{M_n}{\log{n}}=\frac{2}{\log{1/p}}\right)=1,\]where $p=\mathbbm{P}(X_1=Y_1)$ the collision probability, and $-\log\mathbbm{P}(X_1=Y_1)$ is often called the collision entropy or R\'{e}nyi entropy.\\

The result can easily be translated to topological Markov shifts by replacing the stochastic sequences with two points in the shift spaces $\Sigma_X$ and $\Sigma_Y$ with distribution $\mu_X,\mu_Y$, and one can verify a similar convergence law holds for $M_n(\underline{x},\underline{y})$ (see \cite{Dembo1994LimitDO} and \cite{10.1214/aop/1176988492}). In recent works, \cite{barros2019shortest} the authors proved as an improvement of the results in \cite{Dembo1994LimitDO}, for subshift systems with an invariant probability measure $\mu$ admitting good mixing conditions (more precisely, $\alpha$-mixing with exponential decay or $\psi$-mixing with polynomial decay), the shortest distance between the $n$-orbit of a typical \textit{pair} of points $\underline{x},\underline{y}$,
\[M_n\left(\underline{x},\underline{y}\right)
=\sup\{k:x_{i_1+j}=y_{i_2+j}\text{,  $j=0,\dots,k-1$, for some $i_1,i_2\leq n-k$}\}\] converges to a R\'{e}nyi entropy for $\mu\otimes^2-$ almost every $(\underline{x},\underline{y})$ in $\Sigma_A^2$. Similar almost sure convergences are proved for $k-$point-orbits in \cite{barros2021shortest} for all $k\geq2$.

The analogous problem for $M_n(\underline{x})$ is more difficult due to short return phenomenon. For subshifts of finite type, Collet et al in \cite{10.2307/30243633} applied first and second-moment analysis to the counting random variable $N(\underline{x},n,r_n)$, which counts the number of matches of subwords of length $r_n$ among the first $n$ iterates in $\underline{x}$, there exists a constant $H_2$ which is the R\'{e}nyi entropy of a Gibbs measure $\mu$, such that 
\[\lim_{n\rightarrow+\infty}\mu\left(\left|\frac{M_n(\underline{x})}{\log{n}}-\frac{2}{H_2}\right|>\varepsilon\right)=0,\]that is, $\frac{M_n(\underline{x})}{\log{n}}$ converges to $\frac{2}{H_2}$ in probability for typical $\underline{x}$. Then one may ask if this result can be improved to an almost sure convergence, or if the convergence remains valid when the alphabet is countably infinite. The answer is given by the following theorem.

\begin{theorem}\label{Gibbs Cylinder}
For a one-sided subshift system $(\Sigma_A,\sigma,\mu,\mathcal{I})$ admitting a Gibbs measure $\mu$, 
\[\lim_{n\rightarrow+\infty}\frac{M_n(\underline{x})}{\log{n}}=\frac{2}{H_2}\]for $\mu$-almost every $\underline{x}\in\Sigma_A$, where $M_n(\underline{x})$ is defined in \eqref{ineqn:1.1}.
\end{theorem}
The existence and uniqueness of the Gibbs measure with respect to locally H\"{o}lder potentials for finite $\mathcal{I}$ is well-known (see \cite{article} for reference). For $\mathcal{I}$ countably infinite, they are charaterised by theorems in \cite{sarig_1999} and \cite{Sarig03}. Detailed discussion is in \autoref{Thermodynamic Formalism for Gibbs Measures} below.\\

The counterpart of the longest substring matching problem for dynamical systems $(T,X,\mu)$ acting on non-symbolic metric spaces $(X,d)$ investigates the shortest distance between the two $n$-orbits generated by a typical pair of points. To be precise, we care about the following quantity 
\[m_n(x,y):=\min_{0\leq i,j\leq n-1}d\left(T^ix,T^jy\right).\]
There is a dimension-like object for measures, called \textit{correlation dimension} (denoted as $D_2(\mu)$, Definition~\ref{Def:4.1}), which plays a similar role to R\'{e}nyi entropies for symbolic systems. In \cite{barros2019shortest} the authors gave an asymptotic relation between $m_n(x,y)$ and the correlation dimension $D_2(\mu)$ for $\mu\otimes^2$-almost every $(x,y)$, provided good decay of correlations. Later in \cite{barros2021shortest}, this rule is generalised for a typical collection of $k$ points, $(x_1,x_2,\dots,x_k)$, for $k=2,3,\dots$. Again, we extend the investigation to the one-point case. Adopting techniques from \cite{gouezel:hal-03788538}, we will show that for one-dimensional Gibbs-Markov interval maps, there is a similar asymptotic relation between $m_n(x)$ and $D_2(\mu)$, where \[m_n(x):=\min_{0\leq i<j\leq n-1}d\left(T^ix,T^jx\right).\label{1.2}\tag{1.2}\]
\begin{theorem}\label{GM case}
Let $X$ be a closed interval of $\mathbbm{R}$, $(X, T)$ a Gibbs-Markov system and $\mu$ a Gibbs probability measure admitting exponential decay of correlations for $L^1$ against $\mathcal{BV}$ observables. Then if its upper correlation dimension $\overline{D}_2(\mu)$ is bounded from 0,
\[\liminf_{n\rightarrow\infty}\frac{\log{m_n(x)}}{-\log{n}}\geq\frac{2}{\overline{D}_2}\]
for $\mu-$almost every $x$ in the repeller $\Lambda$. If $\mu$ is absolutely continuous with respect to the Lebesgue measure, then \[\lim_{n\rightarrow\infty}\frac{\log{m_n(x)}}{-\log{n}}=\frac{2}{D_2}\]for $\mu$-almost every $x$, and in this case $\underline{D}_2(\mu)=\overline{D}_2(\mu)=1$.
\end{theorem}\begin{remark}
    The proof for this setting is more challenging than that for the symbolic setting because the open balls defined by the Euclidean metric and the cylinders generated by the natural partitions disagree. The analysis of short return to balls is crucial for obtaining the upper bound of $\frac{\log{m_n(x)}}{-\log{n}}$, which is also generally harder than the recurrence analysis of cylinders. The proof for the lower bound relies on the $4-mixing$ property of Gibbs measures proved in Lemma~\ref{lemma:4-mixing}.
\end{remark}
\begin{remark}
    Another difference between the two theorems stated above and theorems proved in \cite{barros2019shortest}, \cite{barros2021shortest} is that in the single point case, obtaining asymptotic upper bounds of $M_n(x)$ and $m_n(x)$ requires good mixing properties. This should be expected due to the fundamental difference between one-point orbits and orbits generated by multiple independent points, the strengthening of assumptions is to ensure the iterates decorrelate fast enough to behave like an independent sequence after a relatively small number of iterations. 
    \end{remark}
\vspace{3mm}
This theorem is applicable to the a range of systems, for example,
\begin{exmp}[\textbf{k-doubling maps}]
    $f:[0,1]\rightarrow[0,1]$, $f(x)=kx$ (mod 1) for $k=2,3,\dots$, and $\mu=Leb$.
\end{exmp}
\begin{exmp}[\textbf{Piecewise affine interval maps}]
   Let $\{a_k\}_k$ be a monotonic sequence with $a_1=1$ and $\lim_ka_k=0$. Then $f:[0,1]\rightarrow[0,1]$ with
    \[f|_{[a_{k+1},a_k)}=\frac{1}{a_k-a_{k+1}}\left(x-a_{k+1}\right)\]satisfies the assumptions above.
\end{exmp}
\begin{exmp}[\textbf{Gauss Map}]
Define the Gauss map $G:[0,1]\rightarrow[0,1]$ by  
\[G(x)=\left\{
\begin{aligned}&\frac{1}{x} \text{ (mod 1)}&x\in\left(0,1\right]\\
&0&x=0
\end{aligned}
\right.
\]
It is a full-branched map. Let $\mu_G$ be the Gauss measure, it is the Gibbs measure for the potential $-\log{DF}$ with density $\frac{d\mu_G}{dLeb}=\frac{1}{(1+x)\log2}$, then \autoref{GM case} holds for $(G,\mu_G)$. 
\end{exmp}
\begin{exmp}[\textbf{An induced map}]
    Let $F$ be the first return function to $[0,\frac{1}{2})$ of a Manneville–Pomeau map $f:[0,1]\rightarrow[0,1]$:
    \[
    f(x)=\left\{
    \begin{aligned}
    &x(1+2^ax^a)&x\in\left[0,\frac{1}{2}\right)\\
    &2x-1&x\in\left[ \frac{1}{2},1\right]
    \end{aligned}
    \right.
    \]
    for $a\in(0,1)$. There exists $\mu_F$ a Gibbs measure with respect to the potential $-\log{DF}$ (see \cite{LIVERANI_SAUSSOL_VAIENTI_1999} or \cite[\S13.2]{HolNicTor})and is absolutely continuous with respect to the Lebesgue measure.
\end{exmp}\vspace{3mm}

\section{Preliminaries for \autoref{Gibbs Cylinder}}\label{Section 2}
We consider the one-sided symbolic dynamics. Let $\mathcal{I}$ be an at most countably infinite alphabet, $A$ a $\mathbbm{N}\times\mathbbm{N}$ transition matrix of $0,1$ entries and \[\Sigma_A=\left\{\underline{x}=(x_0,x_1,x_2,\dots):x_i\in\mathcal{I}, A_{x_i,x_{i+1}}=1\right\}\]be the symbolic space. It is equipped with a metric $d(\cdot,\cdot)=d_{\theta}(\cdot,\cdot)$, for some $\theta\in(0,1)$, and \[d(\underline{x},\underline{y})=\theta^{\underline{x}\land\underline{y}},\hspace{5mm}\underline{x}\land\underline{y}:=\min\left\{j\geq0:x_{j}\neq y_{j}\right\}.\]
Without loss of generality, we may always assume $\theta=e^{-1}$. The shift space in question is denoted by $(\Sigma_A, \sigma, d)$, where the left shift map $\sigma$
\[\sigma:\Sigma_A\rightarrow\Sigma_A,\hspace{4mm}(x_0,x_1,\dots)\mapsto(x_1,x_2,\dots).\]When $\mathcal{I}$ is finite the space $(\Sigma_A,d)$ is compact. A measure $\mu$ is $\sigma$-invariant if for all measurable $E\subseteq\Sigma_A$\[\mu(E)=\mu(\sigma^{-1}E).\] 
In dynamical systems literature, $(\Sigma_A,\sigma)$ is often referred to as a \textit{topological Markov chain}.

\subsection{General Definitions and Lemmas}
First, we need the most basic notion of open sets in the symbolic space.
\begin{definition}[Cylinders]
A \textit{k-cylinder} in $\Sigma_A$ is a subset set of the form
\[[x_0,x_1,\dots,x_{k-1}]=\{\underline{y}\in\Sigma_A:y_i=x_i,\forall i=0,1,\dots,k-1\}.\]The set of all $k$-cylinders is denoted as $\mathcal{C}_k$. For any $\underline{x}\in\Sigma_A$, let $C_k(\underline{x})$ be the $k$-cylinder containing $\underline{x}$, i.e. $C_k(\underline{x})=[x_0,x_1,\dots,x_{k-1}]$. Cylinders are generating in the sense that any subset in $\Sigma_A$ can be represented by a countable union of cylinders. Also, cylinders are also the open balls defined by the symbolic metric $d$.
\end{definition}

\begin{definition}[R\'{e}nyi entropy]
For each $n\in\mathbbm{N}$, $t>0$, define the quantities \[Z_n(t)=\sum_{C\in\mathcal{C}_n}\mu(C)^{1+t}.\]
The \textit{R\'{e}nyi entropy} (with respect to the natural partition given by the alphabet $\mathcal{I}$) of the system is given by \[H_2(\mu)=\lim_{n\rightarrow+\infty}\frac{\log{Z_n(1)}}{-n},\tag{2.2}\label{ineqn:2.2}\]whenever this limit exists, and the generalised R\'{e}nyi entropy function is \[\mathcal{R}_{\mu}(t)=\liminf_{n\rightarrow+\infty}\frac{\log{Z_n(t)}}{-tn}.\]
In the information theory context, this entropy is also called \textit{collision entropy} for Bernoulli systems, as it reflects the probability of two i.i.d. random variables coinciding i.e. $H_2=-\log{\sum_{i}p_i^2}=-\log\mathbbm{P}{(X=Y)}$.\\
R\'{e}nyi entropy does not always exist, especially when the alphabet is not finite. For the finite alphabet case, Haydn and Vaienti proved in \cite[Theorem 1]{haydn_vaienti_2010} that $\mathcal{R}_{\mu}(t)$ converges uniformly on compact subsets of $\mathbbm{R}^{+}$ for all weakly $\psi$-mixing invariant measures, in particular, if $\mu$ is a Gibbs measure, $H_2(\mu)=\mathcal{R}_{\mu}(1)=2P_{top}(\phi)-P_{top}(2\phi)$ where $P_{top}$ is the topological pressure. For infinite alphabet Markov chains, R\'enyi entropy is obtained in \cite{5895075}.
\end{definition}
\begin{definition}{}
The system is said to be $\psi$\textit{-mixing}, if there is some monotone decreasing function $\psi(\cdot):\mathbbm{N}\rightarrow[0,\infty)$ such that for all $n,k$, all $E\in\mathcal{C}_n$, and all $F\in\mathcal{C}^*$, where $\mathcal{C}^*=\bigcup_{j=0}^{\infty}\mathcal{C}_j$, 
\[\left|\frac{\mu(E\cap\sigma^{-n-k}F)}{\mu(E)\mu(F)}-1\right|\leq\psi(k).\tag{2.3}\label{ineqn:2.3}\]
A measure is called \textit{quasi-Bernoulli} if there is some constant $B>1$ such that for any finite words \textbf{i}, \textbf{j}$\in\mathcal{C}^*$
\[\mu([\textbf{ij}])\leq B\mu([\textbf{i}])\mu([\textbf{j}]).\] Automatically, $\psi$-mixing entails the quasi-Bernoulli property. 
\end{definition}
\begin{lemma}\label{exp decay psi mixing}
If the probability measure $\mu$ is $\psi$-mixing with $\psi(\cdot)$ summable, there exists constant $\rho\in(0,1)$ such that $\mu(C)\leq\rho^n$ for all  $C\in\mathcal{C}_n$ and all $n$. 
\end{lemma}
The proof is given for finite alphabet case in \cite{MR1483874} which remains valid for countable alphabet case.
\begin{notation}

For two sequences $\{a_k\}_k,\{b_k\}_k$, the following notation is inherited from \cite[Def 2.9]{10.2307/30243633}.\\Say $a_k\approx b_k$ if $\log{a_k}-\log{b_k}$ is bounded, or equivalently the ratio $\left|\frac{a_k}{b_k}\right|$ is uniformly bounded away from $0$ and $+\infty$.\\ Say $\displaystyle {a_k\preceq b_k}$ if there is $\{c_k\}_k$ such that $a_k\leq c_k$ for all $k$, and $b_k\approx c_k$. \\Both relations are transitive.
\end{notation}

\begin{notation}The following notation is also used.\\
(1) For each $\omega\in\mathbbm{N}$ and any finite word $(x_0,x_1,\dots,x_{k-1})$ of length $k\geq2$, the notation $(x_0,x_1,\dots,x_{k})^\omega$ means the word is repeated $\omega$-times whenever it is allowed. \\
(2) For any $\underline{x}\in\Sigma$, denote the $k$-word starting from position $m$ by $x_m^k=(x_m,x_{m+1},\dots,x_{m+k-1})$.\\
(3) The indicator function of a set $E$ is denoted by $\mathbbm{1}_E$.\\
(4) The expectation and variance of a random variable $X$ are denoted respectively by $\mathbbm{E}[X]$ and $Var[X]$.\\
(5) The cardinality of a set $E$ is denoted by $\#E$.

\end{notation}
\subsection{Thermodynamic Formalism for Gibbs Measures}\label{Thermodynamic Formalism for Gibbs Measures}
In \cite{10.2307/30243633} the authors considered the substring matching problem for a single point in a Markov subshift system, the relevant measure $\mu$ is Gibbsian with respect to a locally H\"older potential. In this section, we will provide conditions and lemmas which enable us to include certain types of countable Markov subshifts. The majority of references for this section can be found in \cite{sarig_1999} and \cite{Sarig03}. 

\begin{definition}\label{Def:2.5}
A potential $\phi:\Sigma_A\rightarrow\mathbbm{R}$ is called \textit{locally H\"{o}lder} if there exists $M_{\phi}>0,$ $\zeta\in(0,1)$ such that for all $k\geq1$, $var_k(\phi)\leq M_{\phi}\zeta^k$ where \[var_k(\phi):=\sup\left\{\left|\phi(\underline{x})-\phi(\underline{y})\right|:x_i=y_i,\,\forall i\leq k-1\right\}.\] 
\end{definition}
\begin{definition}
The system $(\Sigma_A,\sigma)$ is \textit{topologically mixing} if for all $a,b\in \mathcal{I}$, there is $n_0\in\mathbbm{N}$ such that for all $n>n_0$,\[[a]\cap\sigma^{-n}[b]\neq\emptyset.\]
    The system has \textit{big image and preimage} property, if there is a finite subset $S\subseteq \mathcal{I}$ such that for any $a\in\mathcal{I}$, there are $b_1,b_2\in\mathcal{S}$ such that $[b_1\,a\,b_2]\neq\emptyset$. This condition is trivially satisfied for all subshifts of finite type.
\end{definition}
\begin{definition}\label{Def:2.6}
Given a potential $\phi$ on $\Sigma_A$, a $\sigma$-invariant measure $\mu$ is said to be\textit{ Gibbs }if there are constants $c_{\phi}>0$, $P\in\mathbbm{R}$ such that for each $m\in\mathbbm{N}$,  \[c_{\phi}^{-1}\leq\frac{\mu(C_m(\underline{x}))}{\exp{\left(-mP+S_m\phi(\underline{x})\right)}}\leq c_{\phi},\tag{2.4}\label{ineqn:2.4}\]where $S_m\phi(\underline{x})=\sum_{i=0}^{m-1}\phi(\sigma^i\underline{x})$.
\end{definition}

\begin{definition}
Let $\phi$ be a locally H\"{o}lder potential, the \textit{partition functions} with respect to $\phi$ are defined by 
\[P_n(\phi,a):=\sum_{\substack{\sigma^n\underline{x}=\underline{x}\\x_0=a}}e^{S_n(\phi)(\underline{x})}\tag{2.5},\] for each $a\in\mathcal{I}$, and the \textit{Gurevich pressure} $P_G(\phi)$:\[P_G(\phi):=\lim_{n\rightarrow\infty}\frac{1}{n}\log{P_n(\phi,a)}.\tag{2.6}\label{2.6}\]
For topologically mixing countable subshifts, $P_G(\phi)$ exists and is independent of the symbol $a\in\mathcal{I}$ (\cite[Theorem 1]{sarig_1999}).\end{definition} 
\vspace{3mm}
The existence of Gibbs measure for a countable topological Markov subshift is characterised by the following theorem:

\begin{theorem}\label{Sarig CMS}\cite[Theorem 1]{Sarig03},\cite[Theorem 8]{sarig_1999} Let $\phi$ be a locally H\"{o}lder potential, $(\Sigma_A,\sigma,\mathcal{I})$ topologically mixing , then $\phi$ has an invariant Gibbs measure $\mu$ if and only if the system satisfies the big image and preimage property and $P_G(\phi)<\infty$. In particular, let $\mathcal{L}_\phi$ be the Ruelle-Perron-Frobenius operator associated with $\phi$,\[\mathcal{L}_\phi f(\underline{x})=\sum_{\sigma\underline{y}=\underline{x}}e^{\phi(\underline{y})}f(\underline{y}),\] then $\lambda=e^{P_G(\phi)}$ is the eigenvalue of $\mathcal{L}_\phi$ and the eigenfunction $h$ is uniformly bounded from $0$ and $\infty$, also $P_G(\phi)=P$ for $P$ in \eqref{ineqn:2.4}.
\end{theorem}
\begin{remark}
    It is shown by \cite[Theorem 3,Theorem 8]{sarig_1999} that the Gibbs measure given in Theorem~\ref{Sarig CMS} is the unique equilibrium state which realises the equality below
    \[P_G(\phi):=\sup\left\{h_\nu+\int\phi\,d\nu:\nu \text{ is $\sigma$ invariant and }\int\phi\,d\nu>-\infty\right\}.\]
\end{remark}

We will need the following lemmas for the proof of Theorem~\ref{Gibbs Cylinder}. 
\begin{lemma}\label{Lemma:psi-mixing}
    For subshifts of finite type or countable shifts satisfying the assumptions of \autoref{Sarig CMS}, if $\mu$ is the (unique) Gibbs measure with respect to a locally H\"{o}lder potential $\phi$, it has exponential rate $\psi$-mixing for cylinders.
\end{lemma}
\begin{proof}For $\mathcal{I}$ finite, one can verify exponential decay of $\psi$ with \cite[Proposition 1.14]{article}.\\
For countable $\mathcal{I}$, because the corresponding Gibbs measures for two comohologous H\"{o}lder potentials coincide, without loss of generality, one can assume $\mathcal{L}_\phi1=1$, which implies $P_G(\phi)=0$ and the conformal measure $\nu$ identifies with the Gibbs measure $\mu$ on cylinders.\\
Firstly, by locally H\"{o}lder property of $\phi$, there is $M_1>0$ such that \[\left|e^{S_n\phi(\underline{x})-S_n\phi(\underline{y})}-1\right|\leq M_1d(\underline{x},\underline{y})\tag{2.7}\label{ineqn:2.7}\]whenever there is $a\in\mathcal{I}$ such that $\underline{x},\underline{y}\in[a]$\\
Also, as the invariant density is uniformly bounded from 0 and $+\infty$, there is $M_2>0$ such that for each $n$-cylinder $C\in\mathcal{C}_n$,\[M_2^{-1}e^{S_n\phi}\leq\mu(C)\leq M_2e^{S_n\phi}.\label{ineqn:2.8}\tag{2.8}\]
Now define the norm for real-valued function $f$ acting on $\Sigma_A$,
\[\|f\|_L:=\|f\|_{\infty}+D_{\beta}f,\]where $\beta$ is the $\sigma-$algebra generated by $\{\sigma[a]:a\in\mathcal{I}\}$ and 
\[D_{\beta}:=\sup_{b\in\beta}\sup_{x,y\in b}\frac{|f(x)-f(y)|}{d(x,y)}.\]The operator $\mathcal{L}_\phi:Lip_{1,\beta}\rightarrow L$ where the spaces are defined by $Lip_{1,\beta}:=\{f:\Sigma_A\rightarrow\mathbbm{R}:\|f\|_1,D_{\beta}f\leq\infty\}$ and $L:=\{f:\Sigma_A\rightarrow\mathbbm{R}:\|f\|_L<\infty\}$.\\
Consider $E=[e_0,e_1,\dots,e_{n-1}]\in\mathcal{C}_n$ and $F\in\mathcal{C}^*$, as $\mathcal{L}_\phi^*\mu=\mu$,
\begin{align*}
    &\left|\mu(E\cap\sigma^{-(n+k)}F)-\mu(E)\mu(F)\right|
    =\left|\int\mathbbm{1}_E\cdot\mathbbm{1}_F\circ\sigma^{n+k}\,d\mu-\int\mathbbm{1}_E\,d\mu\int\mathbbm{1}_F\,d\mu\right|\\
    &\leq\left|\int\mathbbm{1}_F\left((\mathcal{L}_\phi^{n+k}\mathbbm{1}_E)-\int\mathbbm{1}_E\,d\mu\right)\,d\mu\right|
    \leq\mu(F)\left\|\mathcal{L}_\phi^{n+k}\mathbbm{1}_E-\int\mathbbm{1}_F\,d\mu\right\|_{\infty}\\
    &\leq\mu(F)\left\|\mathcal{L}_\phi^k(\mathcal{L}_\phi^n\mathbbm{1}_E)-\int\mathcal{L}_\phi^n\mathbbm{1}_E\,d\mu\right\|_{L}
\end{align*}
It is a standard fact (see for example \cite[Theorem 1.6]{Aaronson2007LOCALLT} or \cite[Theorem 5]{sarig_1999}) that there are $K_{\phi}>0$ and $\kappa\in(0,1)$ such that 
\[\left\|\mathcal{L}_\phi^k(\mathcal{L}_\phi^n\mathbbm{1}_E)-\int\mathcal{L}_\phi^n\mathbbm{1}_E\,d\mu\right\|_{L}\leq K_{\phi}\kappa^k\|\mathcal{L}_\phi^n\mathbbm{1}_E\|_L.\]\\
\textit{Claim.} $\|\mathcal{L}_\phi^n\mathbbm{1}_E\|_L\leq M_3\mu(E)$ for some $M_3>0$.
\begin{proof}[Proof of claim.] It is easy to see for each $E\in\mathcal{C}_n$ and $\underline{x}\in\Sigma_A$, there is (at most) only one $\underline{z}\in E=[e_0,e_1,\dots,e_{n-1}]$ such that $\sigma^n\underline{z}=\underline{x}$, i.e.
\[\underline{z}=(e_0,\dots,e_{n-1},x_0,x_1,\dots),\]hence by \eqref{ineqn:2.8}
\[\mathcal{L}_\phi^n\mathbbm{1}_E(\underline{x})=\sum_{\sigma^n\underline{y}=\underline{x}}e^{S_n\phi(\underline{y})}\mathbbm{1}_E(\underline{y})=e^{S_n\phi(\underline{z})}\leq M_2\mu(E),\]
and for $\underline{x},\underline{y}\in[b]\in\beta,$ by \eqref{ineqn:2.7}
\begin{align*}\left|\mathcal{L}_\phi^n\mathbbm{1}_E(\underline{x})-\mathcal{L}_\phi^n\mathbbm{1}_E(\underline{y})\right|\leq\sum_{\substack{\underline{z},\underline{w}\in[e_0,\dots,e_{n-1},b]\\\sigma^n\underline{z}=\underline{x},\,\sigma^n\underline{w}=\underline{y}}}e^{S_n\phi(\underline{w})}\left|e^{S_n\phi(\underline{z})-S_n\phi(\underline{w})-1}\right|
\leq M_2\mu(E)M_1d(\underline{x},\underline{y}),
\end{align*}this gives $\mathcal{L}_\phi^n\mathbbm{1}_E\leq\mu(E)M_1M_2$.
and the claim is proved with $M_3=M_2(1+M_1)$.
\end{proof}
The proof of lemma follows from the claim.
\end{proof}

Now we will show that for countable topological Markov shifts, R\'{e}nyi entropy of the Gibbs measure $\mu$ exists and is given by a formula involving the pressure function. The analogous statement for subshift of finite types is mentioned in \cite{haydn_vaienti_2010} and is easy to verify. 
\begin{lemma}\label{lemma:RenyiEntropy}
Under the assumptions of \autoref{Sarig CMS}, for $\mu$ the unique Gibbs measure with respect to $\phi$, the R\'enyi entropy $H_2(\mu)$ exists and is given by
\[H_2(\mu)=\lim_{n\rightarrow+\infty}\frac{\log{Z_n(1)}}{-n}=2P_G(\phi)-P_G(2\phi).\] \end{lemma}
\begin{proof} 
Firstly, $H_2$ is clearly non-negative. By \eqref{ineqn:2.4}, for $B_1:=\sum_{k\geq1}var_k(\phi)$,
\[Z_{n+k}(1)=\sum_{C\in\mathcal{C}_{n+k}}\mu(C)^2\le c_\phi^2\sum_{C\in\mathcal{C}_{n+k}}\exp{\left(S_{n+k}\phi(\underline{x})-(n+k)P_G\right)}\le c_\phi^4e^{2B_1}Z_n(1)Z_k(1),\]
so $-\log Z_n(1)$ is almost subadditive, and the limit $\frac{\log Z_n(1)}{-n}$ exists, in particular, every subsequence converges to the same limit. Suppose $\underline{x}$ is a periodic point with period $k$, then for all $n$, 
\[\mu(C_{nk}(\underline{x}))^2\le Z_{nk}(1)=\sum_{C\in\mathcal{C}_{nk}}\mu(C)^2,\]
\[\liminf_{n\rightarrow\infty}\frac{-2\log c_\phi +2\left(S_{nk}\phi(\underline{x})-nkP_G\right)}{nk}\le\liminf_{n\rightarrow\infty}\frac{\log Z_{n}(1)}{n}.\tag{?}\]
Since $S_{nk}\phi(\underline{x})=nS_k\phi(\underline{x})$ and both $P_G$ and $S_k\phi(x)/k$ are finite, we get $\limsup_n\frac{\log Z_n(1)}{-n}<\infty$.

Combining the BIP property and locally H\"older property with \cite[Lemma 4]{sarig_1999} one can show that 
\[\limsup_{n\rightarrow\infty}\frac{1}{n}\log\sum_{C\in\mathcal{C}_n}\exp\left(\sup_{\underline{x}\in C}2S_n\phi(\underline{x})\right)\leq P_G(2\phi)\]
which implies \[\limsup_{n\rightarrow\infty}\frac{1}{n}\log\sum_{C\in\mathcal{C}}\mu(C)^2\leq \limsup_{n\rightarrow\infty}\frac1nc_\phi^2e^{-2nP_G(\phi)}\sum_{C\in\mathcal{C}_n}\exp\left(\sup_{\underline{x}\in C}2S_n\phi(\underline{x})\right)\le P_G(2\phi)-2P_G(\phi).\]

Also for each $C\in\mathcal{C}_n$, there is at most one $\underline{x}\in C$ such that $\sigma^n\underline{x}=\underline{x}$, thus
\begin{align*}\sum_{C\in\mathcal{C}_n}\mu(C)^2\geq& c_{\phi}^{-1}e^{-2nP_G(\phi)}\sum_{C\in\mathcal{C}_n}\exp{\left(\sup_{\underline{x}\in C}2S_n(\phi(\underline{x}))\right)}\\
\geq &c_{\phi}^{-1}e^{-2nP_G(\phi)}\sum_{\substack{\sigma^n\underline{x}=\underline{x}\in C\\C\in\mathcal{C}_n,\,C\subseteq[a]}}\exp\left(2S_n\phi(\underline{x})\right)=c_{\phi}^{-1}e^{-2nP_G(\phi)}P_n(2\phi,a),\end{align*}
which implies 
\[\liminf_{n\rightarrow\infty}\frac{1}{n}\log\sum_{C\in\mathcal{C}_n}\mu(C)^2\geq P_G(2\phi)-2P_G(\phi).\]
Then putting the inequalities for $\limsup$ and $\liminf$ together,\[H_2=\lim_{n\rightarrow+\infty}\frac{\log{\sum_{C\in\mathcal{C}_n}\mu(C)^2}}{-n}=2P_G(\phi)-P_G(2\phi).\qedhere\]
\end{proof}

\begin{remark}
It is also easy to see that $H_2(\mu)\le 2h_\nu$ for all invariant probability measure $\nu$: for all $\underline{x}$ and all $n\in\mathbbm{N}$,
\[\frac{\log{Z_n(1)}}{-n}\le\frac{2\log\nu(C_n(\underline{x}))}{-n}.\]
By the Shannon-McMillan-Breiman Theorem, the left hand side converges to $2h_{\nu}$ for almost every $\underline{x}$, therefore $\limsup_n \frac{\log Z_n(1)}{-n}\le h_\nu$. So the R\'enyi entropy is finite whenever the measure-theoretic entropy of $\nu$ is finite.
\end{remark}

For simplicity, denote $\alpha=\frac{H_2}{2}$. The following lemma is crucial for approximating the values of $Z_n(t)$. 

\begin{lemma}\label{Lemma:2.4}
For countable Markov shifts satisfying the assumptions of Theorem~\ref{Sarig CMS}, for $\mu$ the invariant Gibbs measure, we have $\alpha>0$ and \[Z_k(1)=\sum_{C\in\mathcal{C}_k}\mu(C)^2\approx e^{-2k\alpha},\]and for each $t>2$,\[Z_k(t-1)=\sum_{C\in\mathcal{C}_k}\mu(C)^t\preceq e^{-tk\alpha}.\]
\end{lemma}
\begin{proof}
    Let $b_n:=\max_{C_\in\mathcal{C}_n}\mu(C)$, then by Lemma~\ref{exp decay psi mixing} $\sum_{C\in\mathcal{C}_n}\mu(C)^2\le b_n\sum\mu(C)\le \rho^n$, hence 
    \[\liminf_{n\rightarrow\infty}\frac{\log Z_n(1)}{-n}\ge\liminf_{n\rightarrow\infty}\frac{-\log b_n}{-n}\ge-\log\rho>0.\]   
    The approximation formulae are from \cite[Lemma 2.13]{10.2307/30243633}. They were originally proved for finite alphabets and the proof remains valid if one combines with \cite[Theorem 1 (IV)]{haydn_vaienti_2010} which holds whenever the relevant measure admits exponential decay of cylinder measures. 
\end{proof}
\vspace{4mm}

\section{Proof of \autoref{Gibbs Cylinder}}
We will use dynamical Borel-Cantelli argument to show separately:
\[\limsup_{n\rightarrow+\infty}\frac{M_n(\underline{x})}{\log{n}}\leq\frac{2}{H_2}\tag{3.1.1}\label{ineqn:3.1.1}\] and \[\liminf_{n\rightarrow+\infty}\frac{M_n(\underline{x})}{\log{n}}\geq\frac{2}{H_2}\tag{3.1.2}\label{ineqn:3.1.2}\]for $\mu$-almost every $\underline{x}\in\Sigma_A$. Together they give \[\lim_{n\rightarrow\infty}\frac{M_n(\underline{x})}{\log{n}}=\frac{2}{H_2}.\]
\subsection{Upper Bound (\ref{ineqn:3.1.1})}
\begin{proof}
Set \[r_n=\frac{1}{\alpha-\varepsilon}\left(\log{n}+\log\log{n}\right).\] 
As $M_n(\underline{x})=r_n$ implies the return time of some iterate of $\underline{x}$ under $\sigma$ to some $r_n$-cylinder is strictly less than $n$, we need to approximate the size of short return sets in the system in order to apply Borel-Cantelli Lemmas to obtain almost everywhere statements. Hence, as in \cite{haydn_vaienti_2010} and \cite{10.2307/30243633}, we intend to solve this by considering different cases of overlapping between $r_n$-substrings in $\underline{x}$.  
\subsubsection*{Overlapping Analysis}\label{Overlapping Analysis}
Let $r_n$ be given; if $r_n$ is not an integer, we simply take the closest integer since when $n\rightarrow+\infty$, $r_n\rightarrow+\infty$ it will not make any difference in terms of limiting behaviours. Let us define the following auxiliary sets.
\[S_{k}(r_n)=\left\{\underline{x}\in\Sigma_A:\sigma^{k}\underline{x}\in C_{r_n}(\underline{x})\right\}.\]
In other words, it is the set of points whose return time of $\underline{x}$ to the $r_n$-cylinder containing itself is $k$, and 
\[\mu\left(\left\{\underline{x}:\exists i,k\text{ such that }d(\sigma^i\underline{x},\sigma^{i+k}\underline{x})\leq e^{-r_n}\right\}\right)\leq\mu\left(\bigcup_{i=0}^{n-1}\bigcup_{k=1}^{n-i-1}\sigma^{-i}S_{k}(r_n)\right) .\tag{3.2}\label{ineqn:3.2}\]
In order to obtain good estimates of $\mu(S_k(r_n))$, we consider three separate cases according to the range of $k$.\\
Let 
\[\Sigma_0=\Sigma_0(n):=\mu\left(\bigcup_{i=0}^{n-1}\bigcup_{k=1}^{\lfloor r_n/2\rfloor}\sigma^{-i}S_{k}(r_n)\right).\]
Similarly, set \[\Sigma_1:=\mu\left(\bigcup_{i=0}^{n-1}\bigcup_{k=\lfloor r_n/2\rfloor+1}^{r_n}\sigma^{-i}S_{k}(r_n)\right),\] and
\[\Sigma_2:=\mu\left(\bigcup_{i=0}^{n-1}\bigcup_{k=r_n+1}^{n-i-1}\sigma^{-i}S_k(r_n)\right).\]
Moreover,
\[\mu(\{M_n>r_n\})\leq\mu(M_n\geq r_n)\leq \Sigma_0+\Sigma_1+\Sigma_2.\tag{3.2'}\label{ineqn:3.2'}\]

\subsubsection*{$\Sigma_0$: return time $1\leq k\leq\lfloor r_n/2\rfloor$}\label{S_0 estimation}
Let $\omega_k=\lfloor \frac{r_n}{k}\rfloor$ and $0\leq\gamma_k<k$ so that $r_n=k\omega_k+\gamma_k$. Then if $\underline{x}\in \sigma^{-i}S_k(r_n)$, $x_j=x_l$ if $j=l$ (mod $k$) for all $j,\,l\in[i,i+r_n+k-1]$, therefore $\sigma^i\underline{x}$ has the following form:
\[\sigma^i\underline{x}=\left(x_{i}^{k},x_{i+k}^{k},
\dots,x_{i+k\omega_k}^k,x_{i+k\omega_k+1}^{\gamma_k},\dots\right)
=\left((x_i,\dots,x_{i+k-1})^{\omega_k+1},x_{i+k\omega_k+1}^{\gamma_k},\dots\right),\]
that is, a $k$-word $(x_i,\dots,x_{i+k-1})$ will be repeated fully for $\omega_k+1$ times, followed by a truncated $\gamma_k$-word with the same initial symbols. Also, for each $k\leq\lfloor r_n/2\rfloor$, there is at least one $\ell=\ell_k\in\left[\lceil r_n/4\rceil,\lfloor r_n/2\rfloor\right]$ such that $\ell$ is a multiple of $k$, meaning that $\underline{x}\in \sigma^{-i}S_k(r_n)\subseteq \sigma^{-i}S_{\ell}(r_n)$ where the $\ell$ word is fully repeated $\omega_{\ell_k}+1\leq5$ times, i.e. we may interpret $\sigma^i\underline{x}$ as\[\sigma^i(\underline{x})=\left(\left(x_i,x_{i+1}\dots,x_{i+\ell_k-1}\right)^{\omega_{\ell_k+1}},x_i^{\gamma_k},\dots\right).\] Therefore for each $k\leq\lfloor r_n/2\rfloor$, by the quasi-Bernoulli property, 
\begin{align*}
\mu\left(\{\underline{x}:\sigma^{i+k}\underline{x}\in C_{r_n}(\sigma^i\underline{x})\}\right)=\mu(S_k(r_n))\leq\mu(S_{\ell_k}(r_n))\leq B^{6}\sum_{C_{\ell_k}\in\mathcal{C}_{\ell}}\mu(C_{\ell})^{\omega_{\ell_k}+1}\rho^{\gamma_{k}}\leq B^6Z_{\ell_k}(\omega_{\ell_k}),
\end{align*}
where $\rho$ is a given by Lemma~\ref{exp decay psi mixing}. As $r_n\leq\ell_k\omega_{\ell_k}\leq r_n+1$, $e^{-\alpha\ell_k\omega_k}\leq e^{-\alpha r_n}$, by definition of $r_n$, \[Z_{\ell_k}(\omega_k)\preceq e^{-\alpha r_n}\leq \exp{(-\log{n}-\log\log{n})}\leq\frac{1}{n\log{n}},\]
For each $i\leq n-1$ and $1\leq k\leq \lfloor r_n/2\rfloor$, we can omit $\sigma^{-i}S_k(r_n)$ if $2k\leq n-i-1$ to avoid overcounting the redundant terms because $\sigma^{-i}S_k(r_n)\subseteq \sigma^{-i}S_{2k}(r_n)$, hence for $\Sigma_0$ we only need to consider points $\underline{x}$ such that $\sigma^i\underline{x}$ has short return time and $i\geq n-r_n$. 
As $r_n$ is in the scale of $\log{n}$, we may choose $n$ large such that $r_n\leq n^{1/2}$ so that by Lemma~\ref{Lemma:2.4},
\begin{align*}\Sigma_0\leq B^6\sum_{k=1}^{\lfloor r_n/2\rfloor}kZ_{\ell_k}(\omega_{\ell_k})
\preceq B^6r_n^2e^{-\alpha r_n}
\leq\frac{B^6r_n^2}{n\log{n}}
\preceq\frac{1}{\log{n}}\tag{3.3}\label{ineqn:3.3}.
\end{align*}

\subsubsection*{$\Sigma_1$: return time $\lfloor r_n/2\rfloor+1\leq k\leq r_n$}\label{S_1 estimation}
In this case, $\underline{x}\in \sigma^{-i}S_k(r_n)$ implies $x_j=x_l$ if $j=l$ (mod $k$) for all $j,l\in[i,i+r_n+k-1]$, hence $\sigma^i\underline{x}$ has the form
\[\sigma^i\underline{x}=\left(x_i^{r_n-k},x_{i+r_n-k}^{2k-r_n},x_{i+k}^{r_n-k},x_{i+r_n}^{2k-r_n}
,x_{i+2k}^{r_n-k}\dots\right)\\=\left(x_i^{r_n-k},x_{i+r_n-k}^{2k-r_n},x_{i}^{r_n-k},x_{i+r_n-k}^{2k-r_n},x_i^{r_n-k}\dots\right)\]that is, the $(r_n-k)$ word starting from $x_i$ is repeated three times, separated by two identical $(2k-r_n)$ words. Hence by Lemma~\ref{Lemma:2.4} and the quasi-Bernoulli property, the following upper bound for $\mu(\sigma^{-i}S_k(r_n))$ holds:
\[\mu(\sigma^{-i}S_k(r_n))=\mu(S_k(r_n))\leq B^6\sum_{\substack{C\in\mathcal{C}_{r_n-k}\\D\in\mathcal{C}_{2k-r_n}}}\mu(C)^3\mu(D)^2=B^6Z_{r_n-k}(2)Z_{2k-r_n}(1).\]
As $\alpha>0$, 
\begin{align*}
\Sigma_1\leq& B^6\sum_{k=\lfloor r_n/2\rfloor+1}^{r_n}(n-k)Z_{r_n-k}(2)Z_{2k-r_n}(1)
\preceq B^6\sum_{k=\lfloor r_n/2\rfloor+1}^{r_n}(n-k)e^{-\alpha3(r_n-k)}e^{-\alpha2(2k-r_n)}\\
=&B^4\sum_{k=\lfloor r_n/2\rfloor+1}^{r_n}(n-k)e^{-\alpha(r_n+k)}
\leq B^6e^{-\frac{3}{2}r_n\alpha}\sum_{k=\lfloor r_n/2\rfloor+1}^{r_n}(n-k)
\preceq r_n ne^{-\frac{3}{2}\alpha r_n}\\
\leq&\frac{nr_n}{(n\log{n})^{3/2}}
\leq\frac{n^{1/2}n}{(n\log{n})^{3/2}}\leq\frac{1}{\log{n}}\tag{3.4}\label{ineqn:3.4},
\end{align*}

\subsubsection*{$\Sigma_2$: return time $r_n+1\leq k\leq n-i-1$}
In this case, $k-r_N\geq1$ and $\underline{x}\in \sigma^{-i}S_k(r_n)$ implies $x_i^{r_n}$, the $r_n$-word starting from position $i$ of $\underline{x}$,  is repeated from the $i+k$ entry without any overlapping with itself, i.e.
\[\sigma^i\underline{x}=\left(x_i^{r_n},x_{i+r_n}^{k-r_n},x_{i+k}^{r_n},\dots\right)
=\left(x_i^{r_n},x_{i+r_n}^{k-r_n},x_{i}^{r_n},\dots\right)\]
then by the $\psi$-mixing condition, \[\mu(\sigma^{-i}S_k(r_n))\leq(1+\psi(k-r_n))\sum_{C\in\mathcal{C}_{r_n}}\mu(C)^2=(1+\psi(k-r_n))Z_{r_n}(1).\]\\therefore
\[\Sigma_2\leq\sum_{k=r_n+1}^{n-1}(n-k)\mu(S_k(r_n))\leq\sum_{k=r_n+1}^{n-1}(n-k)(1+\psi(k-r_n))Z_{r_n}(1).\]
Given that $(1+\psi(k))$ is monotonically decreasing in $k$, 
\begin{align*}
    \Sigma_2\approx &\,e^{-2\alpha r_n}\sum_{k=r_n+1}^{n-1}(n-k)(1+\psi(k-r_n))
    \leq(1+\psi(1))e^{-2\alpha r_n}\sum_{k=r_n+1}^{n-1}(n-k)\\
    \leq&(1+\psi(1))e^{-2\alpha r_n}n^2 \leq\frac{1+\psi(1)}{(\log{n})^2}\leq\frac{1+\psi(1)}{\log{n}}\tag{3.5}\label{ineqn:3.5}
\end{align*}
Then, combining \eqref{ineqn:3.2'}-\eqref{ineqn:3.5}, there is some constant $K_1>0$ independent of $n$ such that \[\mu\left(\{M_n>r_n\}\right)\leq K_1\frac{1}{\log{n}}.\] Using the technique in the proof of \cite[Theorem 5]{barros2019shortest}, picking a subsequence $n_k=e^{\lceil k^2\rceil}$, then for all $k$ large enough,\[\mu(\{M_{n_k}>r_{n_k}\})\leq K_1\frac{1}{k^2},\]then by Borel-Cantelli Lemma, for $\mu$-almost every $\underline{x}\in\Sigma_A$,
\[M_{n_k}(\underline{x})\leq r_{n_k}\] which implies for all $k$ large enough, \[\frac{M_{n_k}(\underline{x})}{\log{n_k}}\leq\frac{1}{\alpha-\varepsilon}\left(1+\frac{\log\log{n_k}}{\log{n_k}}\right).\]Now taking the limsup of the inequality above, and since $M_n(\underline{x})$ is non-decreasing in $n$ for all $\underline{x}$, for each $n$, there is a unique $k$ such that $n_k\leq n< n_{k+1}$ with 
\[\frac{\log{n_k}}{\log{n_{k+1}}}\cdot\frac{M_{n_k}(\underline{x})}{\log{n_{k}}}\leq\frac{M_n(\underline{x})}{\log{n}}\leq\frac{M_{n_{k+1}}(\underline{x})}{\log{n_{k+1}}}\cdot\frac{\log{n_{k+1}}}{\log{n_k}}\tag{3.6}\label{ineqn:3.6},\]
the following inequality holds for all $\varepsilon>0$ small,\[\limsup_{n\rightarrow+\infty}\frac{M_n(\underline{x})}{\log{n}}=\limsup_{n\rightarrow+\infty}\frac{M_{n_k}(\underline{x})}{\log{n_k}}\leq\frac{1}{\alpha-\varepsilon}\]since \[\lim_{k\rightarrow+\infty}\frac{\log{n_{k+1}}}{\log{n_{k}}}=1,\text{ and }
\lim_{k\rightarrow+\infty}\frac{\log\log{n_k}}{\log{n_k}}=0.\]
Then \eqref{ineqn:3.1.1} is proved by letting $\varepsilon\rightarrow0$
\end{proof}
\subsection{Lower bound (\ref{ineqn:3.1.2})}\label{proof of 3.1.2}
\begin{proof}
    We apply a similar second-moment analysis as in the proof of \cite[Theorem 4.1]{10.2307/30243633}. Let \[r_n=\frac{1}{\alpha+\varepsilon}\left(\log{n}+\beta\log\log{n}\right)\]
    for some uniform constant $\beta<0$ to be determined later.
Since  $\sigma^{i+k}\underline{x}\in C_{r_n}(\sigma^i\underline{x})$ if and only if $\underline{x}\in\sigma^{-i}S_k(r_n)$, then we can define the random variable $\mathcal{S}_n$: \[\mathcal{S}_n(\underline{x}):=\sum_{i=0}^{n-2r_n-1}\sum_{k=2r_n}^{n-i-1}\mathbbm{1}_{C_{r_n}(\sigma^i\underline{x})}(\sigma^{i+k}\underline{x})=\sum_{i=0}^{n-2r_n-1}\sum_{k=2r_n}^{n-i-1}\mathbbm{1}_{\sigma^{-i}S_k(r_n)}(\underline{x}),\tag{3.7}\label{rv:3.7}\]
which counts the number of times that $\underline{x}$ is belongs to some $\sigma^{-i}S_k(r_n)$. 
As $M_n(\underline{x})<r_n$ implies for all $0\leq i\leq n-1$, $1\leq k\leq n-i-1$, $\underline{x}\notin \sigma^{-i}S_k(r_n)$, and in particular not in the $\sigma^{-i}S_k(r_n)$ sets with $k\geq 2r_n$, therefore \[\{M_n(\underline{x})<r_n\}\subseteq\{S_n(\underline{x})=0\}\] and by Paley-Zygmund's inequality \cite{PZ1932},
\[\mu(\{M_n(\underline{x})<r_n\})\leq\mu(\{\mathcal{S}_n(\underline{x})=0\})\leq \frac{Var[\mathcal{S}_n]}{\mathbbm{E}[\mathcal{S}_n^2]}\leq\frac{Var[\mathcal{S}_n]}{\mathbbm{E}[\mathcal{S}_n]^2}.\tag{3.8}\label{ineqn:3.8}\]
By definition of $\sigma^{-i}S_k(r_n)$ with $k\geq2r_n$, this set corresponds to the set of points in which an $r_n$-word repeats itself at least once with at least an $r_n$ gap, therefore we have the following lower bound using the $\psi$-mixing property, 
\begin{align*}\mu\left(\{C_{r_n}(\sigma^i\underline{x})=C_{r_n}(\sigma^{i+k}\underline{x})\}\right)=&\mu\left(\sigma^{-i}S_k(r_n)\right)=\sum_{C\in\mathcal{C}_{r_n}}\mu(C\cap\sigma^{-k}C)\\
\geq&(1-\psi(k-r_n))\sum_{C\in\mathcal{C}_{r_n}}\mu(C)^2\geq(1-\psi(r_n))Z_{r_n}(1),\end{align*}
therefore 
\[\mathbbm{E}[\mathcal{S}_n]\geq\frac{1}{2}(1-\psi(r_n))(n-2r_n)^2Z_{r_n}(1)\tag{3.9}\label{ineqn:3.9}.\]
Next, we need to consider \[\mathbbm{E}[\mathcal{S}_n^2]=\sum_{i,j=0}^{n-2r_n-1}\sum_{k=2r_n}^{n-i-1}\sum_{l=2r_n}^{n-j-1}\mu(\sigma^{-i}S_k(r_n)\cap \sigma^{-j}S_l(r_n))\tag{3.10}\label{ineqn:3.10}.\]
Define the index set \[F:=\left\{(i,j,k,l)\in\mathbbm{N}^4:0\leq i,\,j\leq n-2r_n-1,2r_n\leq k\leq n-i-1,\,2r_n\leq l\leq n-j-1\right\},\]
then 
\[\mathbbm{E}[\mathcal{S}_n^2]=\sum_{(i,j,k,l)\in F}\mu(\sigma^{-i}S_k(r_n)\cap \sigma^{-j}S_l(r_n)),\tag{3.10'}\label{ineqn:3.10'}\]
and the cardinality of $F$ satisfies
\[\#F=\left(\sum_{i=2r_n}^{n-2r_n-1}n-i\right)\left(\sum_{j=2r_n}^{n-2r_n-1}n-j\right)\leq\frac{1}{4}(n-2r_n)^4.\]
Define the counting function by
\[\theta:F\rightarrow\mathbbm{N},\hspace{4mm}\theta(i,j,k,l)=\sum_{\substack{a\in\{i,i+k\}\\b\in\{j,j+l\}}}\mathbbm{1}_{(a-r_n,a+r_n)}(b),\]i.e. it counts the occurrences that two indices in $\{i,j,i+k,j+l\}$ are $r_n$-close to each other; $\theta>0$ translates to overlapping between some $r_n$ words, e.g. $|i-j|<r_n$ implies the $r_n$ word $x_i^{r_n}$ overlaps with the $r_n$ word $x_j^{r_n}$, and both $r_n$-strings are repeated later.\\
By our definition of $\mathcal{S}_n$, for each quadruple $(i,j,k,l)$, necessarily $k,l\geq2r_n$ which implies
\[\theta(i,j,k,l)\leq2,\hspace{4mm}\text{ $\forall(i,j,k,l)\in F$},\]which allows us to split \eqref{ineqn:3.10'} again into 3 components,
\[\mathbbm{E}[\mathcal{S}_n^2]=\left(\sum_{F_0}+\sum_{F_1}+\sum_{F_2}\right)\mu\left(\sigma^{-i}S_k(r_n)\cap \sigma^{-j}S_l(r_n)\right),\]
where $F_t=\{(i,j,k,l)\in F:\theta(i,j,k,l)=t\}$. \\Clearly,
\[\#F_0\leq\#F\leq\frac{1}{4}(n-2r_n)^4.\]For each $(i,j,k,l)\in F_1$, if we fix any three indices, for example, if $i,j,k$ are fixed, $j+l$ can be $r_n$-close to either $i$ or $i+k$ as it is automatically $2r_n$- apart from $j$, hence there are at most $4r_n$ choices for the remaining index $l$. Hence \[\#F_1\leq 2r_n(n-2r_n)^3,\]
and similarly if we fix any two of $i,j,k,l$ in $F_2$, there are at most $2r_n^2$ choices for the remaining two indices, therefore\[\#F_2\leq2r_n^2(n-2r_n)^2.\]
\subsubsection*{Contributions of indices in $F_0$:}
We will consider the sum over indices in $F_0$ first. Since $(i,j,k,l)\in F_0$ implies no overlapping, $\underline{x}\in \sigma^{-i}S_k(r_n)\cap \sigma^{-j}S_l(r_n)$ implies $x_i^{r_n}=x_{i+k}^{r_n}$ and $x_j^{r_n}=x_{j+l}^{r_n}$ while the symbols in these two $r_n$-strings are independent, e.g. when $i+k<j$, $\underline{x}$ has the following form:
\begin{align*}\sigma^i\underline{x}=&\left(x_i^{r_n},\dots,x_{i+k}^{r_n},x_{i+k+r_n},x_{i+k+r_n+1},\dots,x_j^{r_n},\dots,x_{j+l}^{r_n}\dots\right)
=\left(x_i^{r_n},\dots,x_i^{r_n},\dots,x_j^{r_n},\dots,x_j^{r_n},\dots\right).\end{align*}
Hence by $\psi$-mixing property
\[\mu(\sigma^{-i}S_k(r_n)\cap \sigma^{-j}S_l(r_n))\leq(1+\psi(\gamma_{i\,j\,k\,l}))^3Z_{r_n}(1)^2,\tag{3.11}\label{ineqn:3.11}\]where \[\gamma_{i\,j\,k\,l}=\min\left\{|a-b|-r_n:{a,b\in\{i,j,i+k,j+l\}}\right\}.\]
Let $F_0'\subseteq F_0$ be defined as \[F_0':=\{(i,j,k,l)\in F_0:\gamma_{ijkl}\geq r_n\},\]and $F_0'':=F_0\setminus F_0'.$ Notice also that $\#(F_0'')\leq 2r_n(n-2r_n)^3$.\\
Define the notation for any $G\subseteq F$,\[\mathbbm{E}[\mathcal{S}_n^2|G]:=\sum_{i,j,k,l\in G}\mu\left(\sigma^{-i}S_k(r_n)\cap \sigma^{-j}S_l(r_n)\right).\]
Then, using \eqref{ineqn:3.11}, \begin{align*}\mathbbm{E}[\mathcal{S}_n^2|F_0']&=\sum_{i,j,k,l\in F_0'}\mu\left(\sigma^{-i}S_k(r_n)\cap \sigma^{-j}S_l(r_n)\right)\leq(n-2r_n)^4(1+\psi(g))^3Z_{r_n}(1)^2,\end{align*}
By Lemma~\ref{Lemma:psi-mixing} $\psi(r_n)\leq r_n^{-1}$ for all $n$ large enough, then
\begin{align*}
    &\left(1+\psi(r_n)\right)^3-\left(1-\psi(r_n)\right)^2\leq\left(1+\psi(r_n)\right)^3-\left(1-\psi(r_n)\right)^3\\&=2\,\psi(r_n)\left(2+2\psi(r_n)^2+1-\psi(r_n)^2\right)
    \leq2r_n^{-1}(3+r_n^{-2})
    \leq8r_n^{-1}\leq\frac{1}{\log{n}}\tag{3.12}\label{ineqn:3.12}
\end{align*} using \eqref{ineqn:3.9} with \eqref{ineqn:3.12}, as $\#(F_0')\leq\frac{1}{4}(n-2r_n)^4$, for some constant $K_2>0$.
\begin{align*}
\frac{\mathbbm{E}[\mathcal{S}_n^2|F_0']-\mathbbm{E}[\mathcal{S}_n]^2}{\mathbbm{E}[\mathcal{S}_n]^2}&\leq\frac{(n-2r_n)^4Z_{r_n}(1)^2\left((1+\psi(r_n))^3-(1-\psi(r_n))^2\right)}{(n-2r_n)^4(1-\psi(r_n))^2Z_{r_n}(1)^2}\\
&\leq\frac{(1+\psi(r_n))^3-(1-\psi(r_n))^3}{(1-\psi(r_n))^2}\leq K_2\frac{1}{\log{n}},\tag{3.13}\label{ineqn:3.13}
\end{align*}

And for the sum over $F_0''$, the term $1+\psi(\gamma_{ijkl})$ in \eqref{ineqn:3.11} is uniformly bounded above by $1+\psi(0)$, and $1-\psi(r_n)\geq\frac{1}{2}$ for all $n$ sufficiently large, therefore 
\[\frac{\mathbbm{E}[\mathcal{S}_n^2|F_0'']}{\mathbbm{E}[\mathcal{S}_n]^2}\approx\frac{r_n(n-2r_n)^3(1+\psi(0))^3Z_{r_n}(1)^2}{(1-\psi(r_n))^2(n-2r_n)^4Z_{r_n}(1)^2}    \preceq \frac{r_n}{n-2r_n}\]
hence for some $K_3>0$ and all $n$ sufficiently large,
\[\frac{\mathbbm{E}[\mathcal{S}_n^2|F_0'']}{\mathbbm{E}[\mathcal{S}_n]^2}\leq K_3\frac{1}{\log{n}} \tag{3.14}\label{ineqn:3.14}\]

\subsubsection*{Contributions of indices in $F_1$:}
Next, for $(i,j,k,l)\in F_1$, without loss of generality, suppose only $|i-j|=r<r_n$, $i<j$ and $i+k<j+l$. The other cases are treated exactly the same since the order of the $r_n$-strings does not have any effects on estimations of the upper bounds for $\mu(\sigma^{-i}S_k(r_n)\cap \sigma^{-j}S_l(r_n))$. \\
An $\underline{x}\in \sigma^{-i}S_k(r_n)\cap \sigma^{-j}S_l(r_n)$ means $x_{i+r}=x_j,$
$x_{i+r+1}=x_{j+1},$ $\dots$, $x_{i+r_n}=x_{j+r}$, so $\sigma^i\underline{x}$ has the following form:
\begin{align*}
\sigma^i\underline{x}=&\left(x_i^{r_n-r},x_j^{r_n},\dots,x_i^{r_n},\dots,x_j^{r_n},\dots\right)=\left(x_i^{r},x_{j}^{r_n-r},x_{j+r_n}^r,\dots,x_i^{r},x_{j}^{r_n-r},\dots,x_{j}^{r_n-r},x_{j+r_n}^r,\dots\right).\end{align*}then using the quasi-Bernoulli property and Lemma~\ref{Lemma:2.4}, for $B>1$ the relevant quasi-Bernoulli constant,
\begin{align*}\mu(\sigma^{-i}S_k(r_n)\cap \sigma^{-j}S_l(r_n))&\leq B^{10}\sum_{\substack{A,B\in\mathcal{C}_r\\ C\in\mathcal{C}_{r_n-r}}}\mu(A)^2\mu(B)^2\mu(C)^3=C^{10}Z_{r}(1)^2Z_{r_n-r}(2)\\
&\preceq B^{10}e^{-4\alpha r_n}e^{-3\alpha(r_n-r)}\leq B^{10}e^{-3\alpha r_n}.\end{align*}
Recall that $r_n=\frac{1}{\alpha+\varepsilon}\left(\log{n}+\beta\log\log{n}\right)$ and
\[e^{-\alpha r_n}=\left(n(\log{n})^{\beta}\right)^{-\frac{\alpha}{\alpha+\varepsilon}}=n^{-\frac{\alpha}{\alpha+\varepsilon}}(\log{n})^{-\frac{\alpha\beta}{\alpha+\varepsilon}}\]
hence for all $n$ large enough such that \[\frac{n^{\frac{\alpha}{\alpha+\varepsilon}}}{(n-2r_n)}\leq1\tag{3.15}\label{ineqn:3.15},\] by \eqref{ineqn:3.9} and \eqref{ineqn:3.15} above, as $\beta<0$,
\begin{align*}
    \frac{\mathbbm{E}[\mathcal{S}_n^2|F_1]}{\mathbbm{E}[\mathcal{S}_n]^2}&\preceq\frac{2r_n(n-2r_n)^3B^{10}e^{-3\alpha r_n}}{(n-2r_n)^4(1-\psi(r_n))^2Z_{r_n}(1)^2}\\
    &\approx\frac{r_n e^{-3\alpha r_n}}{(n-2r_n)e^{-4\alpha r_n}}
    =\frac{1}{\alpha+\varepsilon}\frac{\log{n}+\beta\log\log{n}}{(n-2r_n)e^{-\alpha r_n}}\\
    &\preceq\frac{\log{n}}{(n-2r_n)e^{-\alpha r_n}}=\frac{(\log{n})^{1+\beta\frac{\alpha}{\alpha+\varepsilon}}n^{\frac{\alpha}{\alpha+\varepsilon}}}{n-2r_n}\\
    &\leq(\log{n})^{\left(1+\beta(1-\frac{\varepsilon}{\alpha+\varepsilon})\right)}.
\end{align*}
For all $0<\varepsilon\leq\alpha$ such that $1-\frac{\varepsilon}{\alpha+\varepsilon}\geq\frac{1}{2}$, one can choose \[\beta=-4,\]hence \[1+\beta\left(1-\frac{\varepsilon}{\alpha+\varepsilon}\right)\leq-1\]which is sufficient to conclude that for some constant $K_4$,\[\frac{\mathbbm{E}[\mathcal{S}_n^2|F_1]}{\mathbbm{E}[\mathcal{S}_n]^2}\leq K_4\frac{1}{\log{n}}.\tag{3.16}\label{ineqn:3.16}\]
\vspace{3mm}
\subsubsection*{Contributions of indices in $F_2$:}
Finally for indices in $F_2$, it is enough to know that for any $\underline{x}\in\sigma^{-i}S_k(r_n)\cap \sigma^{-j}S_l(r_n)$ there is some $r_n$-word repeated twice, so we can bound the measure of $\sigma^{-i}S_k(r_n)\cap \sigma^{-j}S_l(r_n)$ for each $(i,j,k,l)\in F_2$ by
\[B^4\sum_{C\in\mathcal{C}_{r_n}}\mu(C)^2\approx B^4e^{-2\alpha r_n}\tag*{by Lemma~\ref{Lemma:2.4}}.\]
Then for $\beta=-4$ and all $n$ verifying \eqref{ineqn:3.15}, as $(1-\psi(r_n))$ is uniformly bounded from below, by \eqref{ineqn:3.9},
\begin{align*}
    \frac{\mathbbm{E}[\mathcal{S}_n^2|F_2]}{\mathbbm{E}[\mathcal{S}_n]^2}&\approx\frac{2r_n^2(n-2r_n)^2B^{4}e^{-2\alpha r_n}}{(n-2r_n)^4(1-\psi(r_n))^2e^{-4\alpha r_n}}\\
    &\approx\frac{r_n^2}{(n-r_n)^2e^{-2\alpha r_n}}\preceq\frac{n^{\frac{2\alpha}{\alpha+\varepsilon}}}{(n-2r_n)^2}(\log{n})^{2+2\beta\frac{\alpha}{\alpha+\varepsilon}}\\
    &\leq(\log{n})^{2(1+\beta(1-\frac{\varepsilon}{\alpha+\varepsilon}))}\leq(\log{n})^{-2},
\end{align*}
and it follows that for some constant $K_5>0$, \[\frac{\mathbbm{E}[\mathcal{S}_n^2|F_2]}{\mathbbm{E}[\mathcal{S}_n]^2}\leq K_5\frac{1}{(\log{n})^2}.\tag{3.17}\label{ineqn:3.17}\]
Then, combining \eqref{ineqn:3.8} \eqref{ineqn:3.10'} \eqref{ineqn:3.13}-\eqref{ineqn:3.17}, there is some constant $K_6>0$ such that \[\mu(\{M_n<r_n\})\leq K_6\frac{1}{\log{n}},\]
hence we can repeat the trick of picking a subsequence $n_k=\lceil e^{k^2}\rceil$, and apply Borel-Cantelli Lemma to the sum $\sum_{k=1}^{\infty}\mu(\{M_{n_k}<r_{n_k}\})<+\infty$, which means for all $k$ large enough,
\[\frac{M_{n_k}(\underline{x})}{\log{n_k}}\geq\frac{1}{\alpha+\varepsilon}\left(1-\frac{4\log\log{n_k}}{\log{n_k}}\right),\]then taking the liminf on both sides and apply the arguments which validate \eqref{ineqn:3.6}, 
\[\liminf_{n\rightarrow+\infty}\frac{M_n(\underline{x})}{\log{n}}=\liminf_{n\rightarrow+\infty}\frac{M_{n_k}(\underline{x})}{\log{n_k}}\geq\frac{1}{\alpha+\varepsilon},\]
then \eqref{ineqn:3.1.2} is verified by letting $\varepsilon\rightarrow0$.
\end{proof}

\section{Gibbs-Markov interval maps.}\label{Interval Maps}
This section is inspired by \cite{barros2019shortest}, \cite{gouezel:hal-03788538} and \cite{HolNicTor}, and we will prove an asymptotic behaviour for $m_n(x)$ for $\mu$-almost every $x$, where the dynamics are interval maps with a Gibbs-Markov structure. Let us first define the measure theoretic dimension object.
\begin{definition}\label{Def:4.1}Suppose $T:X\rightarrow X$ is a measurable map with respect on the probability space $(X,\mu)$. 
Define the \textit{upper} and \textit{lower correlation dimension} of $\mu$ by \[\underline{D}_2(\mu)=\liminf_{r\rightarrow0}\frac{\log\int\mu(B(x,r))\,d\mu(x)}{\log{r}},\]\[\overline{D}_2(\mu)=\limsup_{r\rightarrow0}\frac{\log\int\mu(B(x,r))\,d\mu(x)}{\log{r}}\]respectively, and simply write $D_2(\mu)$ when the two limits coincide. \\
Clearly, $D_2(m)=1$ for $m$ the Lebesgue measure. \end{definition}

\begin{definition}\label{boundedlocalcomplexity}
Say the metric space $(X,d)$ satisfies the \textit{bounded local complexity} condition if there exists $C_0\in\mathbbm{N}$ such that for each $r>0$, there is $k(r)<\infty$, and $\{x^{r}_1,x^{r}_2,\dots,x_{k(r)}^{r}\}\subseteq X$ such that 
\[X\subseteq\bigcup_{p=1}^{k(r)}B(x^{r}_p,r)\]and each $x\in X$ belongs to at most $C_0$ elements of $\{B(x_p^{r},2r)\}_{p=1}^{k(r)}$. \\
Any compact subset of $\mathbbm{R}$ has bounded local complexity: compact implies totally bounded which gives $k(r)<\infty$ for all $r$ and $C_0$ can be chosen to be 4 because one can choose an $r$-net such that $d(x_i^r,x_j^r)\geq r$ for $i\neq j\in\{1,\dots,k(r)\}$. 
\end{definition}

\begin{definition}\textbf{(Piecewise expanding interval map)}\label{Def:4.2}
Let $X$ be a closed interval in $\mathbbm{R}$, $T:X\rightarrow X$ is a \textit{piecewise 
 expanding interval map} if there is a (at most) countable partition $\mathcal{P}=\{I_1,I_2,\dots\}$ for $T$ such that $T$ is differentiable on each $I_k$ and there is a uniform constant $\gamma>0$ such that $|DT_{I_k}|>\gamma$. An $n$-cylinder $[x_0,x_1,\dots,x_{n-1}]$ with respect to the partition $\mathcal{P}$ is given by\[[x_0,x_1,\dots,x_{n-1}]=\bigcap_{i=0}^{n-1}T^{-i}P_{x_i}.\]
    Then a point $x\in X$ has a symbolic representation $\underline{x}=\left(x_0,x_1,\dots\right)$ if $x\in\bigcap_{i=0}^{\infty}T^{-i}P_{x_i}$. Denote $\mathcal{P}_n:=\bigvee_{j=0}^{n-1}T^{-j}\mathcal{P}$.\\
We require the following conditions on the probability preserving system $(T,\mu)$.
\begin{itemize}
    \item 
    Say $T:X\rightarrow X$ has exponential decay of correlation for $\mathcal{BV}$ against $L^1$ observables, where $\mathcal{BV}:=\{\varphi\in L^1(m):\varphi \text{ has bounded variation.}\}$, if there is $\beta:\mathbbm{N}\rightarrow \mathbbm{R}$ with $\beta(n)=C_1e^{-c_1n}$ for some $C_1$, $c_1>0$, and for all $f,g:X\rightarrow\mathbbm{R}$, $f\in\mathcal{BV}$ and $g\in L^1$, 
\[\left|\int f\cdot g\circ T^n\,d\mu-\int f\,d\mu\int g\,d\mu\right|\leq\|f\|_{\mathcal{BV}}\|g\|_{1}\beta(n)\] where the norm $\|f\|_{\mathcal{BV}}:=\|f\|_{1}+TV(f)$, and $TV(f)$ stands for the total variation of $f$. In particular, if $f$ is an indicator function of some measurable $A\subseteq X$, $\|f\|_{\mathcal{BV}}=2$ and $\|f\|_1\leq1$.
    \item $T$ has bounded distortion, that is, there is $C_{bd}>0$ such that for all $x,y\in I\in\mathcal{P}_n$,
    \[\frac{1}{C_{bd}}\leq\frac{DT^n(x)}{DT^n(y)}\leq C_{bd}.\]
\end{itemize}

\end{definition}

\begin{definition} \textbf{({Gibbs-Markov} maps)}\label{Gibbs structure}
The map $T:X\rightarrow X$ is said to be \textit{Gibbs-Markov} if

\begin{itemize}
        \item $T(P)$ is a union of elements in $\mathcal{P}$ and $T|_{P}$ is injective for all $P\in\mathcal{P}$.
        \item Let $d_s$ be the symbolic metric as, $d_s(\underline{x},\underline{y})=e^{-\underline{x}\land \underline{y}}$. The map $\log{g|_P}$ is Lipschitz with respect to $d_s$ for each $P\in\mathcal{P}$, where $g=\frac{dm}{d(m\circ T)}$.
        \item  When the partition does in fact contains infinitely many continuous components, we assume there is $\delta_0$ such that $|T^n(I)|\geq \delta_0$ for all $I\in\mathcal{P}_n$, for all $n\geq1$. Otherwise, $\delta_0$ is a positive constant such that $|I|\geq\delta_0$ for all $I\in\mathcal{P}$. This is also known as the big image property. 
    \end{itemize} 
   
\end{definition}
\begin{remark}
For a Gibbs-Markov map defined above one can check that there is a measure $\mu$ verifying the inequality for $n$-cylinders as \eqref{ineqn:2.4} holds for $T$, and because the system is conjugate to a topological Markov shift satisfying the conditions of \autoref{Sarig CMS}, the invariant density $h=d\mu/dm$ is bounded from 0 and $+\infty$.
\end{remark}
\begin{definition}\label{Def:repeller}
    The \textit{repeller} $\Lambda$ of $T$ is defined as the following:
    \[\Lambda:=\left\{x\in I:T^k(x)\in \bigcup_{P\in\mathcal{P}}P\text{ for all }k\geq0\right\}.\]
\end{definition}
\vspace{5mm}
Before the proof our quantitative result for limiting behaviour of $m_n$, following the convention in \cite{gouezel:hal-03788538}, define the following quantities, \[\alpha(n)=(\log{n})^{2}\] 
\[m_n^{\leq}(x):=\min_{\substack{0\leq i<j<n\\|i-j|\leq \alpha(n)}}d\left(T^ix,T^jx\right),\]
\[m_n^{>}(x):=\min_{\substack{0\leq i<j<n\\|i-j|> \alpha(n)}}d\left(T^ix,T^jx\right),\]
\[m_n^{\gg}(x):=\min_{\substack{0\leq i\leq n/3\\2n/3\leq j<n}}d\left(T^ix,T^jx\right),\]
Then, \autoref{GM case} can be rephrased as the following.
\begin{theorem} 
\label{Upper Bound}
Let $T:X\rightarrow X$ be a piecewise expanding map as Definition~\ref{Def:4.2} with a Gibbs-Markov structure defined in Definition~\ref{Gibbs structure}, and $\mu$ its invariant Gibbs measure admitting exponential decay of correlations for $\mathcal{BV}$ against $L^1$ observables, then one has
\[\limsup_{n\rightarrow\infty}\frac{\log{m_n^{>}(x)}}{-\log{n}}\leq\frac{2}{\underline{D}_2}\tag{4.1.1}\label{ineqn:4.1.1}.\]
If $\mu$ is absolutely continuous to Lebesgue measure $m$, $\underline{D}_2=\overline{D}_2=1$, and
\[\limsup_{n\rightarrow\infty}\frac{\log{m_n^{\leq}(x)}}{-\log{n}}\leq\frac{2}{D_2},\tag{4.1.2}\label{ineqn:4.1.2}\]
for $\mu$-almost every $x\in\Lambda$. 
Then as $m_n(x)=\min\{m_n^{\leq},m_n^>(x)\}$, 
\[\limsup_{n\rightarrow+\infty}\frac{\log{m_n}(x)}{-\log{n}}\leq\frac{2}{D_2}=2.\]\end{theorem}
\begin{theorem}\label{Lower bound}
For all Gibbs measures $\mu$, for $\mu$-almost every $x\in\Lambda$,
\[\liminf_{n\rightarrow\infty}\frac{\log{m_n^{\gg}(x)}}{-\log{n}}\geq\frac{2}{\overline{D}_2}.\tag{4.2}\label{ineqn:4.2}\]Since $-\log{m_n}\geq-\log{m_n^{\gg}},$
\[\lim_{n\rightarrow\infty}\frac{\log{m_n(x)}}{-\log{n}}\geq\frac{2}{\overline{D}_2}\]for $\mu$-almost every $x\in\Lambda$.\\
\end{theorem}
Then these two theorems together imply \autoref{GM case}, i.e. if $\mu$ is absolutely continuous to the Lebesgue measure $m$,
\[\lim_{n\rightarrow\infty}\frac{\log{m_n(x)}}{\log{n}}=2\]$\mu$-almost every $x$.
\begin{remark}    The proof of \eqref{ineqn:4.1.1} uses ideas from the proof of \cite[Proposition 1.10]{gouezel:hal-03788538}, whereas the proof of \eqref{ineqn:4.1.2} requires estimating the measures of sets of short return points. Along the proof one will see also that \eqref{ineqn:4.1.1} and \eqref{ineqn:4.2} hold for all Gibbs invariant measures with exponential decay of correlations and $D_2(\mu)> 0$. \\
    For Gibbs acip $\mu$, the correlation dimension is well defined in the sense that $\overline{D}_2(\mu)=\underline{D}_2(\mu)=1$, because the invariant density with respect to Lebesgue measure $m$ is uniformly bounded hence $D_2(\mu)=D_2(m)$.\\
    The following lemmas are analogous to \cite[Lemma 3.1, Lemma 3.2]{gouezel:hal-03788538}, which can be proved by replacing their $\rho_p^{(r)}$ functions with indicator functions.
\end{remark}

\begin{lemma}\label{lemma:4.2}
For all $x,y\in X$, let $\mathbbm{1}_{p,r}:=\mathbbm{1}_{B(x_p^r,2r)}$, if $X$ has bounded local complexity defined in Definition~\ref{boundedlocalcomplexity},
\[\mathbbm{1}_{B(x,r)}(y)\leq\sum_{p=1}^{k(r)}\mathbbm{1}_{p,r}(x)\mathbbm{1}_{p,r}(y)\leq C_0\mathbbm{1}_{B(x,4r)}(y)\]
\end{lemma}
\begin{lemma}\label{lemma:4.3}
The following equations hold.
\[
\left\{
\begin{aligned}
\limsup_{r\rightarrow0}\frac{\log{\sum_{p=1}^{k(r)}\left(\int\mathbbm{1}_{p,r}d\mu\right)^2}}{\log{r}}&=\overline{D}_2(\mu)\\
\liminf_{r\rightarrow0}\frac{\log{\sum_{p=1}^{k(r)}\left(\int\mathbbm{1}_{p,r}d\mu\right)^2}}{\log{r}}&=\underline{D}_2(\mu),
\end{aligned}
\right.\tag{4.3.1}\label{ineqn:4.3.1}\]
which means for any $\varepsilon>0$, there is $r_0>0$ such that for all $0<r<r_0$,
\[r^{\overline{D}_2+\varepsilon}\leq\sum_{p=1}^{k(r)}\left(\int\mathbbm{1}_{p,r}\,d\mu\right)^2\leq r^{\underline{D}_2-\varepsilon}.\tag{4.3.2}\label{ineqn:4.3.2}\]
\end{lemma}

Also, for simplicity, the following definition is introduced.
\begin{definition}A term is said to be \textit{admissible}\label{admissible}, which is a notion introduced by the authors of \cite{gouezel:hal-03788538}, if it has the form $r^{-k}g(n)$, for some $k\geq0$ and a function $g$ which decays in $n$ faster than any polynomial of $n$, hence for any $k\in\mathbbm{N}$, by \eqref{ineqn:4.3.2} and choosing the scale of $r$ as in \eqref{ineqn:4.4} below we can bound any admissible error by $\mathcal{O}(n^{-k})$ for all $n$ large.\end{definition}
Now we can prove \eqref{ineqn:4.1.1}.
\subsection*{Upper Bound}
\begin{proof}[Proof of \eqref{ineqn:4.1.1}]Let $\varepsilon,r>0$ be given, in particular, $r$ should be small enough that it verifies \eqref{ineqn:4.3.2}. Define the random variable $\mathcal{S}_n^{>}$,
\[\mathcal{S}_n^>(x):=\sum_{\substack{0\leq i< j<n\\|i-j|>\alpha(n)}}\mathbbm{1}_{p,r}(T^ix)\mathbbm{1}_{p,r}(T^jx).\] 
By Lemma~\ref{lemma:4.2}, $\{m_n^>(x)\leq r\}\subseteq\{\mathcal{S}_n^>\geq 1\}$.
Therefore, by Markov's inequality and decay of correlation,
\begin{align*}
&\mu(x:\mathcal{S}^>_n(x)\geq1)\leq\mathbbm{E}[S_n^>]
=\sum_{\substack{0\leq i<j<n\\j-i>\alpha(n)}}\sum_{p=1}^{k(r)}\int \mathbbm{1}_{p,r}(T^ix)\mathbbm{1}_{p,r}(T^jx)\,d\mu(x)\\
&\leq\sum_{\substack{0\leq i<j<n\\j-i>\alpha(n)}}\sum_{p=1}^{k(r)}\left(\left(\int\mathbbm{1}_{p,r}\,d\mu\right)^2+\beta(\alpha(n))\|\mathbbm{1}_{p,r}\|_{\mathcal{BV}}\|\mathbbm{1}_{p,r}\|_1\right)\\
    &\leq\sum_{0\leq i<j<n}r^{\underline{D}_2-\varepsilon}+\sum_{0\leq i<j<n}\sum_{p=1}^{k(r)}\|\mathbbm{1}_{p,r}\|_{\mathcal{BV}}\|\mathbbm{1}_{p,r}\|_1\beta(\alpha(n))\tag*{by \eqref{ineqn:4.3.2}}\\
    &\leq n^2r^{\underline{D}_2-\varepsilon}+C_1e^{-c_1(\log{n})^2}n^2\sum_{p=1}^{k(r)}2\mu(B(x_p,2r))\\
    &\leq n^2r^{\underline{D}_2-\varepsilon}+2C_0C_1e^{-c_1(\log{n})^2}n^2.
\end{align*}

The last inequality follows from the definition of bounded local complexity\[\sum_{p=1}^{k(r)}\mu(B(x_p,2r))=\int\sum_{p=1}^{k(r)}\mathbbm{1}_{p,r}(y)\,d\mu(y)\le \int C_0\,d\mu(y).\]
Then as $e^{-c_1(\log{n})^2}$ decays faster than any polynomial of $n$, picking
\begin{equation*}r=r_n=\exp{\left({-\frac{2+2\varepsilon}{\underline{D}_2-\varepsilon}}(\log{n}+\log\log{n})\right)}\leq n^{-\frac{2+2\varepsilon}{\underline{D}_2-\varepsilon}} \tag{4.4} \label{ineqn:4.4}\end{equation*}

for all $n$ sufficiently large,
\[n^2r_n^{\underline{D}_2-\varepsilon}+2C_0C_1n^2e^{-c_1(\log{n})^2}\leq n^{-2\varepsilon}+C_1e^{-c_1(\log{n})^2}r_n^{-C_2}.\] 
The second term on the right is admissible by definition,  meaning that there is some constant $C_2>0$ such that for all $n$ large enough that $n^{-\varepsilon}\leq\frac{1}{\log{n}}$:
\[\mu(m_n^>(x)\leq r)\leq\mathbbm{E}[S_n^>]\leq C_2 n^{-\varepsilon}\leq\frac{C_2}{\log{n}}.\] Therefore, eventually for $\mu$-almost every $x$, 
\[\frac{\log{m^{>}_{n_k}}(x)}{-\log{n_k}}\leq\frac{2+\varepsilon}{\underline{D}_2-\varepsilon}\left(1+\frac{\log{\log{n_k}}}{\log{n_k}}\right).\]Although $m^{>}_n$ is not monotonically increasing, for each $n\in[n_s,n_{s+1}]$,
\[-\log{m'_{n_{k+1}}(x)}:=-\log{\min_{\substack{0\leq i<j<n_{k+1}\\j-i>\alpha(n_k)}}}\geq -\log{m^>_{n}(x)}\]\[-\log{m''_{n_k}(x)}:=-\log{\min_{\substack{0\leq i<j<n_{k}\\j-i>\alpha(n_{k+1})}}}\leq-\log{m^>_n(x)},\]and one can show that for $\mu$-almost every $x$, for all $k$ large,
\[\frac{-\log{m'_{n_{k}}(x)}}{\log{n_k}}\leq\frac{2+\varepsilon}{\underline{D}_2-\varepsilon}\left(1+\frac{\log{\log{n_k}}}{\log{n_k}}\right),\]and
\[\frac{-\log{m''_{n_{k}}(x)}}{\log{n_k}}\leq\frac{2+\varepsilon}{\underline{D}_2-\varepsilon}\left(1+\frac{\log{\log{n_k}}}{\log{n_k}}\right),\]then the limit can be passed to the whole tail of $m_n^>$, and \eqref{ineqn:4.1.1} is proved.
\end{proof}
\begin{remark}
Note also that for $|i-j|$ large, the measure of the sets $\{x:d\left(T^ix,T^jx\right)<r\}$ scales like $r^{D_2}$ which is similar to the behaviour of the sequence matching problem in the symbolic setting, and this matches our intuition because $D_2(\mu)$ is analogous to $H_2$ in many ways. \\
To prove \eqref{ineqn:4.1.2}, we are dealing with iterates of $x$ which return to an $r$-neighborhood of itself within $\alpha(n)$ units of time, that is, we need to approximate the measure of some short return sets as those $S_m(k)$ sets defined in \autoref{Overlapping Analysis} for the symbolic case. But we cannot expect a similar upper bound for short returns as in \eqref{ineqn:3.3} or \eqref{ineqn:3.4}. This is because for symbolic structures, an $r_n$-cylinder is itself an $r_n$ open ball with respect to the symbolic metric, so analysing the returns is equivalent to analysing the repetition of letters in cylinders and one does not need to consider the case that two iterates  $\sigma^i\underline{x}$, $\sigma^j\underline{x}$ are close to the boundaries of two open balls with a common boundary but they belong to different cylinders. 
\\For interval maps, although the Gibbs-Markov structure prescribes a natural partition hence a way to define cylinders, metric balls and symbolic cylinders are different object so one need to take more caution and include the case that two points belong to different cylinders $U,V\in\mathcal{P}_n$ but they accumulate on a common boundary of $U,V$ with distance smaller than the contraction scale of $n$-cylinders.
\end{remark} 
\vspace{5mm}
The measures of the short return sets are approximated by the following lemma for absolutely continuous Gibbs measures.
\begin{lemma}\label{Hol}\cite[Lemma 3.4]{HolNicTor}
Define the sets \[\mathcal{E}_n(\epsilon):=\left\{x\in X:|x-T^nx|\leq\epsilon\right\}.\]Then for $T$ satisfying Gibbs-Markov property and the invariant Gibbs measure $\mu$ absolutely continuous with respect to the Lebesgue measure $m$ with exponential decay of correlation for $\mathcal{BV}$ against $L^1$ observables, there is some constant $C_3$ such that for all $n\in\mathbbm{N}$ and $\epsilon$ small enough,\[\mu\left(\mathcal{E}_n(\epsilon)\right)\leq C_3\epsilon.\] 
\end{lemma}
\begin{proof}
The original lemma states that $m\left(\mathcal{E}_n(\epsilon)\right)=\mathcal{O}( \epsilon)$, and since $m$ is equivalent to $\mu$ on $\Lambda$ with $d\mu/dm$ bounded away from 0 and $+\infty$, there is a uniform constant $C_{\mu}$ such that \[C_{\mu}
^{-1}m(A)\leq \mu(A)\leq C_{\mu}m(A)\]for each measurable $A$, therefore there is some $C_3>0$ such that
\[\mu\left(\mathcal{E}_n(\epsilon)\right)\leq C_3\epsilon.\]
    
\end{proof}

\begin{proof}[Proof of \eqref{ineqn:4.1.2}]
Define the random variable $\mathcal{S}^{\leq}_n$ by
\[\mathcal{S}_n^{\leq}(x):=\sum_{i=0}^{n-1}\sum_{k=1}^{\alpha(n)\land (n-i-1)}\mathbbm{1}_{B(T^ix,r)}(T^{i+k}x),\]
where $a\land b=\min\{a,b\}$. \\
As $\left\{x:\mathbbm{1}_{B(T^ix,r)}(T^{i+k}x)=1\right\}\subseteq \{x:T^ix\in \mathcal{E}_k(r)\}= T^{-i}\mathcal{E}_k(r)$, using the Markov inequality and Lemma~\ref{Hol} we obtain the following bound
\begin{align*}
\mu(\mathcal{S}^{\leq}_n\geq 1)\leq\mathbbm{E}_m[\mathcal{S}_n^{\leq}]\leq\sum_{i=0}^{n-1}\sum_{k=1}^{\alpha(n)\land (n-i-1)}\mu(T^{-i}\mathcal{E}_k(r))\leq n\alpha(n) C_3r
\end{align*}
Pick $r=r_n$ as in \eqref{ineqn:4.4}, for all $n$ large enough such that 
\[\alpha(n)=(\log{n})^2\leq n^{\frac{\varepsilon}{2-\varepsilon}},\hspace{5mm}n^{-\frac{\varepsilon}{2-\varepsilon}}\leq\frac{1}{\log{n}}. \] 
As the invariant density $d\mu/dm$ is uniformly bounded, $D_2(\mu)=1<2$, one has
\[\mu\left(x:m_n(x)\leq r_n\right)\leq C_3n^{1+\frac{\varepsilon}{2-\varepsilon}}r\leq n^{\frac{2}{2-\varepsilon}}C_3n^{-\frac{2+2\varepsilon}{\underline{D}_2-\varepsilon}}\leq C_3n^{-\frac{\varepsilon}{2-2\varepsilon}}\leq\frac{C_3}{\log{n}}.\]
Therefore, by picking a subsequence $n_{k}=\lceil e^{k^2}\rceil$, by Borel-Cantelli Lemma, we have that $\mu\{x:m_{n_{k}}\leq r_{n_{k}} \text{ for infinitely many $k$}\}=0$,  for $\mu$-almost every $x$, for all $k$ large enough, \[m_{n_k}^{\leq}(x)\geq r_{n_{k}},\] and one can apply the arguments used to validate \\eqref{ineqn:3.1.1} to obtain \eqref{ineqn:4.1.2} for $\mu$-almost every $x$. 
\end{proof}
\begin{remark}
The condition that $\mu$ is an acip may be not sharp; if $\Tilde{\mu}$ is another Gibbs measure with exponential decay of correlation and satisfying $\Tilde{\mu}(\mathcal{E}_k(\epsilon))=\mathcal{O}(\epsilon)$, then \autoref{Upper Bound} remains valid.
\end{remark}

As in the symbolic case the proof for the lower bound of $\frac{\log{m_n(x)}}{-\log{n}}$ is slightly more complicated and also requires a second-moment computation which exploits the following notion of mixing.
\begin{lemma}\label{lemma:4-mixing}
       A Gibbs-Markov interval map $(T,\mu)$ has exponential $4-mixing$, that is, for $a<b\leq c$ in $\mathbbm{N}$, there are $C_1',c_1'>0$ such that for any $f_1, f_2\in\mathcal{BV}$, $g_1,g_2\in L^{\infty}$,  such that
\[\left|\int f_1\cdot f_2\circ T^a\cdot g_1\circ T^b\cdot g_2\circ T^c\,d\mu-\int f_1\cdot g_1\circ T^a\,d\mu\int f_2\cdot g_2\circ T^{c-b}\,d\mu\right|\leq C_1'e^{-c_1'(b-a)}.\]The constant $C_1'$ depends on the functions $f_1,f_2,g_1,g_2$.\\
In particular, for any given $r>0$, $0\leq p,q\leq k(r)$, $f_1=g_1=\mathbbm{1}_{p,r},$ $f_2=g_2=\mathbbm{1}_{q,r}$, the constant $C_1'=C_1'(f_1,f_2,g_1,g_2)$ does not depend on $r$.
\end{lemma}
\begin{proof}
    Consider the transfer operator $\mathcal{L}$ associated with the H\"{o}lder potential $\varphi$, that acts on the space of functions of bounded variation of $X$, $\mathcal{BV}=\mathcal{BV}(X)$, \\
    \[\mathcal{L}=\mathcal{L}_{\varphi}:\mathcal{BV}\rightarrow\mathcal{BV},\hspace{4mm}\mathcal{L}_{\varphi}f(x)=\sum_{Ty=x}e^{\varphi(y)}f(y).\]
  
Let $\nu$ be the conformal measure of $\mathcal{L}$ and $h$ the invariant density, $\frac{d\mu}{d\nu}=h$. 
By the following well-known fact, (see for example \cite[(3)]{cmp/1103941781}) for mixing Gibbs-Markov maps, there are $C_{\mathcal{BV}}>0$, $\kappa\in(0,1)$ such that for any $f\in\mathcal{BV}$,
    \[\left\|\mathcal{L}^nf-h\int f\,d\nu\right\|_{\mathcal{BV}}\leq C_{\mathcal{BV}}\cdot\kappa^n\|f\|_{\mathcal{BV}}\]for all $f\in\mathcal{BV}$, there is
    \begin{align*}
    &\left|\int f_1\, f_2\circ T^a \,g_1\circ T^b\, g_2\circ T^c\,d\mu-\int f_1\,f_2\circ T^a\,d\mu\int g_1\,g_2\circ T^{c-b}\,d\mu\right|\\
    =&\left|\int \mathcal{L}^{b-a}(hf_1\,f_2\circ T^a)g_1\circ T^a\,g_2\circ T^{c-b+a}\,d\nu\right.
    -\left.\int\mathcal{L}^a(hf_1)f_2\,d\nu\int h\,g_1\circ T^a\,g_2\circ T^{c-b+a}\,d\nu\right|\\
    =&\left|\int\left(\mathcal{L}^{b-a}(\mathcal{L}^a(hf_1)f_2)-h\int\mathcal{L}^a(hf_1)f_2\,d\nu\right)g_1\circ T^a\,g_2\circ T^{c-b+a}\,d\nu\right|\\
    \leq&\left\|\mathcal{L}^{b-a}(\mathcal{L}^a(hf_1)\, f_2)-h\int \mathcal{L}^a(hf_1)\, f_2\,d\nu\right\|_1\|g_1\|_{\infty}\|g_2\|_{\infty}\\
    \leq&C_{\mathcal{BV}}\cdot\kappa^{b-a}\|\mathcal{L}^a(hf_1)\, f_2\|_{\mathcal{BV}}\|g_1\|_{\infty}\|g_2\|_{\infty},\label{star}\tag{$\star$}
    \end{align*}where the first equality holds by the invariance of $\mu$.
    As $\mathcal{L}^ahf_1$ is of bounded variation because $h$ is of bounded variation, and the product of functions in $\mathcal{BV}$ has bounded variation, the first part of the lemma is proved.\\
    [3mm]
Now we deal with the case $f_1=g_1=\mathbbm{1}_{p,r},$ $f_2=g_2=\mathbbm{1}_{q,r}$ and find a suitable upper bound for 
\[\left\|\mathcal{L}^ah\mathbbm{1}_{p,r}\cdot \mathbbm{1}_{q,r}\right\|_{\mathcal{BV}}=\left\|\mathcal{L}^ah\mathbbm{1}_{p,r}\cdot \mathbbm{1}_{q,r}\right\|_{1}+TV(\mathcal{L}^ah\mathbbm{1}_{p,r}\cdot \mathbbm{1}_{q,r}).\]As $\mathcal{L}$ is a positive operator, by its duality, $\|\mathcal{L}^ah\mathbbm{1}_{p,r}\cdot\mathbbm{1}_{q,r}\|_1=\int\mathbbm{1}_{p,r}\cdot\mathbbm{1}_{q,r}\circ T^a\,d\mu=\int\mathbbm{1}_{p,r}\mathbbm{1}_{q,r}\circ T^a\,d\mu\leq1$.
Also, for any function $u\in\mathcal{BV}$,
\[\|u\|_{\infty}\leq \inf_{x\in X} |u(x)|+TV(u)\leq\inf_{x\in X}|u(x)|\mu(X)+\|u\|_{\mathcal{BV}}\leq\|u\|_1+\|u\|_{\mathcal{BV}},
\]
then by Lasota-Yorke inequality for Gibbs Markov systems, there are $\kappa'\in(0,1)$ and $C'_{\mathcal{BV}}>0$ such that \[\|\mathcal{L}^ah\mathbbm{1}_{p,r}\|_{\mathcal{BV}}\leq\kappa'^a\|h\mathbbm{1}_{p,r}\|_{\mathcal{BV}}+C'_{\mathcal{BV}}\|h\mathbbm{1}_{p,r}\|_1\leq \kappa'^a(1+2\|h\|_{\infty})+C'_{\mathcal{BV}}.\]
therefore, let $C''_{\mathcal{BV}}=(1+2\|h\|_{\infty})+C'_{\mathcal{BV}}$,
\begin{align*}
&TV(\mathcal{L}^ah\mathbbm{1}_{p,r}\cdot \mathbbm{1}_{q,r})\leq\|\mathcal{L}^ah\mathbbm{1}_{p,r}\|_{\infty}TV(\mathbbm{1}_{q,r})+\|\mathbbm{1}_{q,r}\|_{\infty}\|\mathcal{L}^ah\mathbbm{1}_{p,r}\|_{\mathcal{BV}}\\
&\leq\left(\|\mathcal{L}^ah\mathbbm{1}_{p,r}\|_{\mathcal{BV}}+1\right)\cdot2+1\cdot\|\mathcal{L}^ah\mathbbm{1}_{p,r}\|_{\mathcal{BV}}
\leq3\cdot C''_{\mathcal{BV}}+2,
\end{align*}which is a uniform constant that only depends on the operator $\mathcal{L}$, and combining this with \eqref{star}, one obtains $C_1'=C_1'(\mathbbm{1}_{p,r},\mathbbm{1}_{q,r},\mathbbm{1}_{p,r},\mathbbm{1}_{q,r})=\mathcal{O}(1)$.
\end{proof}

\subsection*{Lower Bound \eqref{ineqn:4.2}}
\begin{proof}[proof of \eqref{ineqn:4.2}]
Let $\varepsilon>0$ small be given. Consider the quantity $m_n^{\gg}$ and the random variable $\mathcal{S}_n^{\gg}$:
\[m_n^{\gg}(x):=\min_{\substack{0\leq i\leq n/3\\2n/3\leq j<n}}d\left(T^ix,T^jx\right),\hspace{4mm}\mathcal{S}_n^{\gg}(x):=\sum_{\substack{0\leq i\leq n/3\\2n/3\leq j<n}}\sum_{p=1}^{k(r)}\mathbbm{1}_{p,r}(T^ix)\mathbbm{1}_{p,r}(T^jx).\]
By Lemma \ref{lemma:4.2}, $m_n^{\gg}(x)>4r$ implies for all pairs of $0\leq i\leq \frac{n}{3}$, $\frac{2n}{3}\leq j<n$, if for some $p$, $d(T^ix,x_p^r)<2r$, then $d(T^jx,x_p^r)\geq 2r$ hence $\mathcal{S}_n^{\gg}(x)=0$. By Paley-Zygmund inequality,
\[\mu\left(m_{n}>4 r\right) \leq \mu\left(x:\mathcal{S}_n^{\gg}(x)=0\right)\leq\frac{\mathbbm{E}[(\mathcal{S}_n^{\gg})^2]-\mathbbm{E}[\mathcal{S}_n^{\gg}]^2}{\mathbbm{E}[\mathcal{S}_n^{\gg}]^2}\]
Using decay of correlation and invariance, 
\begin{align*}
\mathbbm{E}[\mathcal{S}_n^{\gg}(x)]
=\sum_{\substack{0\leq i\leq n/3\\2n/3\leq j<n}}\sum_p\int\mathbbm{1}_{p,r}(T^ix)\mathbbm{1}_{p,r}(T^jx)\,d\mu(x)\le\left(\frac{n}{3} \right)^2\sum_p\left(\left(\int\mathbbm{1}_{p,r}\,d\mu(x)\right)^2\pm 2\beta\left(n/3\right)\right)\tag{4.6}\label{ineqn:4.6}
\end{align*}
where $a=b\pm c$ means $a\in[b-c,b+c]$.
Consider 
\[\left(\mathcal{S}_n^{\gg}(x)\right)^2=\sum_{i,j}\sum_{s,t}\sum_{p,q}\mathbbm{1}_{p,r}(T^ix)\mathbbm{1}_{p,r}(T^jx)\mathbbm{1}_{q,r}(T^sx)\mathbbm{1}_{q,r}(T^tx).\] As in the proof of symbolic case, we will split this sum in terms of the distance between the indices $i,j,s,t$. Recall that \[\alpha(n)=(\log{n})^2.\]
Let $F$ be the collection of all possible quadruples of indices $(i,j,s,t)$, and define the counting function 
\[\tau:F\rightarrow \mathbbm{N}\cup\{0\},\hspace{4mm}\tau(i,j,s,t)=\sum_{\substack{a\in\{i,s\}\\b\in\{j,t\}}}\mathbbm{1}_{[a-\alpha(n),a+\alpha(n)]}(b),\]
then $\tau\leq2$ since $i,j$ and $s,t$ are at both at least $\frac{n}{3}$ iterates apart, this allows us to split $F$ into $F_m:=\{(i,j,s,t)\in F:\tau=m\}$ for $m=0,1,2$. Obviously, the following upper bounds hold for the cardinality of each $F_m$, 
\[\#F_m\leq\left(2\alpha(n)\right)^m \left(\frac{n}{3}\right)^{4-m}\tag{4.7}\label{cardbound}.\]
Recall the notation \[\mathbbm{E}[\left(\mathcal{S}_n^{\gg}\right)^2|F_m]=\sum_{(i,j,s,t)\in F_m}\sum_{p,q}\int\mathbbm{1}_{p,r}(T^ix)\mathbbm{1}_{p,r}(T^jx)\mathbbm{1}_{q,r}(T^sx)\mathbbm{1}_{q,r}(T^tx)\,d\mu(x),\]
also for simplicity, let us denote \[R_p=\int\mathbbm{1}_{p,r}\,d\mu=\mu(B(x_p,2r)).\]
\subsubsection*{Contribution of indices in $F_0$:}
For each $(i,j,s,t)\in F_0$, without loss of generality, suppose $i+\alpha(n)<s$ and $j+\alpha(n)<t$, as the alternative cases can be treated equally by exchanging the roles of $i,s$ or $j,t$ and makes no difference to the calculation. As $\min\{j,t\}-\max\{i,s\}\geq\frac{n}3$, by Lemma~\ref{lemma:4-mixing} and invariance, one obtains the following upper bound for each such quadruple $(i,j,s,t)$:\\
\begin{align*}
    &\sum_{p,q}\int\mathbbm{1}_{p,r}(T^ix)\mathbbm{1}_{p,r}(T^jx)\mathbbm{1}_{q,r}(T^sx)\mathbbm{1}_{q,r}(T^tx)\,d\mu(x)\\
    &=\sum_{p,q}\int\mathbbm{1}_{p,r}\mathbbm{1}_{q,r}\circ T^{s-i}\mathbbm{1}_{p,r}\circ T^{j-i}\mathbbm{1}_{q,r}\circ T^{t-i}\,d\mu\\
    &\leq C_1'e^{-c_1'\frac{n}{3}}k(r)^2+\sum_{p,q}\int\mathbbm{1}_{p,r}\mathbbm{1}_{q,r}\circ T^{s-i}\,d\mu\int\mathbbm{1}_{p,r}\mathbbm{1}_{q,r}\circ T^{t-j}\,d\mu\\
    &\leq C_1'e^{-c_1'\frac{n}{3}}k(r)^2+\sum_p\sum_q\left(R_pR_q+2\beta(\alpha(n))\right)^2\\
    &\leq C_1'e^{-c_1'\frac{n}{3}}r^{-2C_0'}+8\beta(\alpha(n))r^{-2C_0'}+\sum_{p,q}\left(R_pR_q\right)^2.
\end{align*}
Where the last inequality holds as $R_p,R_q\leq 1$ for any $p,q$, and by \cite[Lemma 3.3]{gouezel:hal-03788538} $k(r)\leq r^{-C_0'}$ for some $C_0'=4\log C_0$. Any term in the inequality above involving $\beta(\alpha(n))$ or $C_1'e^{-c_1'\frac{n}{3}}$ is admissible, hence for each $k\in\mathbbm{R}$ it is bounded by $\mathcal{O}(n^{-k})$ for all $n$ sufficiently large, and now we shall pick \[r=r_n=n^{-\frac{2-4\varepsilon}{\overline{D}_2+\varepsilon}},\] then by \eqref{ineqn:4.3.2} \[r^{\overline{D}_2+\varepsilon}\leq\sum_p\left(\int\mathbbm{1}_{p,r}\,d\mu\right)^2=\sum_pR_p^2\tag{4.8}\label{ineqn:4.8}.\]Therefore, the contribution of indices in $F_0$ is bounded from above up to an admissible error by
\begin{align*}
    \left(\frac{n}{3}\right)^4\sum_p\sum_q R_p^2R_q^2\leq\left(\frac{n}{3} \right)^4\sum_pR_p^2\sum_qR_q^2=\left(\frac{n}{3}\right)^4\left(\sum_{p}R_p^2\right)^2,
\end{align*}
combining with \eqref{ineqn:4.6}, up to an admissible error term,
\begin{align*}
    &\mathbbm{E}[(\mathcal{S}_n^{\gg})^2|F_0]-\mathbbm{E}[\mathcal{S}_n^{\gg}]^2\leq\mathcal{O}(n^{-\varepsilon})
\end{align*}
also by \eqref{ineqn:4.6}, as $2\beta(\frac{n}3)$ is admissible we can bound it by $n^{-3}$,
\[\mathbbm{E}[\mathcal{S}_n^{\gg}]^2\geq \left(\frac{n}{3} \right)^4\left(r_n^{\overline{D}_2+\varepsilon}-2\beta(\frac{n}{3})\right)\ge \left(\frac{n}{3} \right)^4(n^{-2-4\varepsilon}-n^{-3})^2\approx n^{8\varepsilon}\]  allowing us to conclude that there is some constant $C_4>0$:
\begin{align*}
    \frac{\mathbbm{E}[(\mathcal{S}_n^{\gg})^2|F_0]-\mathbbm{E}[\mathcal{S}_n^{\gg}]^2}{\mathbbm{E}[\mathcal{S}_n^{\gg}]^2}\leq\frac{C_4}{n^{\varepsilon}}.\tag{4.9}\label{ineqn:4.9}
\end{align*}
\subsubsection*{Contributions of indices in $F_1$:}
Now we will deal with the indices in $F_1$. Without loss of generality, suppose $|i-s|\leq\alpha(n)$, $i<s$ and $j<t$, the other cases can be treated by exchanging the roles of $i,s$ or $j,t$. By Lemma~\ref{lemma:4-mixing},
\begin{align*}
&\sum_{p,q}\int\mathbbm{1}_{p,r}(T^ix)\mathbbm{1}_{p,r}(T^jx)\mathbbm{1}_{q,r}(T^sx)\mathbbm{1}_{q,r}(T^tx)\,d\mu(x)\\
&=\sum_{p,q}\int\mathbbm{1}_{p,r}\mathbbm{1}_{q,r}\circ T^{s-i}\mathbbm{1}_{p,r}\circ T^{j-i}\mathbbm{1}_{q,r}\circ T^{t-i}\,d\mu\tag*{by invariance}\\
&\leq\sum_{p,q}\left(\int\mathbbm{1}_{p,r}\mathbbm{1}_{q,r}\circ T^{t-j}\,d\mu\right)\int\mathbbm{1}_{p,r}(x)\mathbbm{1}_{q,r}(T^{s-i}x)\,d\mu(x)+C_1'e^{-c_1'\frac{n}{3}} k(r)^2\\
&\leq\sum_{p,q}\left(R_pR_q+2\beta(\alpha(n))\right)\int\mathbbm{1}_{p,r}(x)\mathbbm{1}_{q,r}(T^{s-i}x)\,d\mu(x)+C_1'e^{-c_1'\frac{n}{3}} r^{-2C_0'},
\end{align*}and this can be bounded by the following up to an admissible error using Cauchy-Schwarz inequality,
\begin{align*}
&\sum_{p,q}\int R_pR_q\mathbbm{1}_{p,r}(x)\mathbbm{1}_{q,r}(T^{s-i}x)\,d\mu(x)\\
&=\int\sum_{p}R_p\mathbbm{1}_{p,r}(x)\sum_qR_q\mathbbm{1}_{q,r}(T^{s-i})\,d\mu(x)\\
&\leq\left(\int\left(\sum_{p}R_p\mathbbm{1}_{p,r}(x)\right)^2\,d\mu\right)^{-\frac{1}{2}}\left(\int\left(\sum_{q}R_q\mathbbm{1}_{q,r}\circ T^{s-i}\right)^2\,d\mu\right)^{-\frac{1}{2}}\\
&=\int\left(\sum_pR_p\mathbbm{1}_{p,r}\right)^2\,d\mu.\tag*{by symmetry and invariance}
\end{align*}
As there are at most $C_0$ non-zero terms of $\mathbbm{1}_{p,r}(x)$ for any $x\in X$, by the following inequality for $a_1,\dots,a_m\geq0$,\[\left(a_1+a_2+\dots+a_m\right)^2\leq m\left(a_1^2+a_2^2+\dots,a_m^2\right)\] together with the fact that $(\mathbbm{1}_{p,r})^2\leq\mathbbm{1}_{p,r}$, the sum above can be bounded by 
\begin{align*}
    \int C_0\sum_{p}R_p^2\mathbbm{1}_{p,r}\,d\mu=\sum_p C_0R_p^2\int\mathbbm{1}_{p,r}\,d\mu=C_0\sum_pR_p^3.
\end{align*}
As $\left(a_1+\dots+a_m\right)^{\frac{2}{3}}\leq\sum_{k=1}^{m}a_k^{\frac{2}{3}},$ clearly $\sum_ka_k\leq\left(\sum_ka_k^{2/3}\right)^{3/2}$, then we get up to an admissible error and some constant $C_5$, for all $n$ large such that $\alpha(n)\leq n^{\varepsilon}$, as $(\sum_pR_p^2)^{-\frac{3}{2}}4\beta(\alpha(n))$ and $\left(\sum_pR_p^2\right)^{-\frac{1}{2}}\beta(\alpha(n))$ are both admissible errors,
\begin{align*}
\frac{\mathbbm{E}[(\mathcal{S}_n^{\gg})^2|F_1]}{\mathbbm{E}[\mathcal{S}_n^{\gg}]^2}\leq&\frac{2\alpha(n)(\frac{n}{3})^3C_0\left(\sum_pR_p^2\right)^{\frac{3}{2}}}{(\frac{n}{3})^4(\sum_pR_p^2-2\beta(\alpha(n)))^2}\\
=&\frac{6\alpha(n)(\frac{n}{3})^3C_0\left(\sum_pR_p^2\right)^{\frac{3}{2}}}{n\left(\left(\sum_pR_p^2\right)^{1/2}-4\beta(\alpha(n))\left(\sum_pR_p^2\right)^{-1/2}-4\beta(\alpha(n))^2\left(\sum_pR_p^2\right)^{-3/2}\right)}\\
\leq&\frac{6C_0\alpha(n)}{n\left((r^{\overline{D}_2+\varepsilon})^{1/2}-\mathcal{O}(n^{-1})\right)}=\frac{6C_0\alpha(n)}{n\left((n^{-1+2\varepsilon}-\mathcal{O}(n^{-1})\right)}\tag*{by (\ref{ineqn:4.8})}\\
\leq&\frac{C_5n^{\varepsilon}}{n\cdot n^{-1+2\varepsilon}}=\frac{C_5}{n^{\varepsilon}}.\tag{4.10}\label{ineqn:4.10}
\end{align*}
\vspace{5mm}
\subsubsection*{Contribution of indices in $F_2$:}
Finally, let us consider indices $(i,j,s,t)$ such that $|i-s|,|j-t|\leq\alpha(n)$. By Lemma~\ref{lemma:4.2}, $\sum_q\mathbbm{1}_{q,r}(T^sx)\mathbbm{1}_{q,r}(T^tx)\leq C_0$ for any $x$, therefore for each $i,j,s,t$ in $F_2$,
\begin{align*}
    &\sum_{p,q}\int\mathbbm{1}_{p,r}(T^ix)\mathbbm{1}_{p,r}(T^jx)\mathbbm{1}_{p,r}(T^sx)\mathbbm{1}_{p,r}(T^tx)\,d\mu(x)\\
    &\leq C_0\sum_p\int\mathbbm{1}_{p,r}(T^ix)\mathbbm{1}_{p,r}(T^jx)\,d\mu(x)\\
    &\leq C_0\sum_pR_p^2+C_0\beta(\frac{n}{3})k(r),
\end{align*}therefore, as $\# F_2\leq\frac{4}{9}\alpha(n)^2n^2$, by our choice of $r_n$ in \eqref{ineqn:4.8}, up to an admissible error there is some constant $C_6$ such that,
\begin{align*}
    \frac{\mathbbm{E}[(\mathcal{S}_n^{\gg})^2|F_2]}{\mathbbm{E}[S_n^{\gg}]^2}&\leq\frac{4\alpha(n)^2(\frac{n}{3})^2C_0\sum_pR_p^2}{(\frac{n}{3})^4(\sum_pR_p^2-2\beta(\alpha(n)))^2}\\
    &= \frac{36C_0\alpha(n)^2}{n^2\left(\sum_pR_p^2-4\beta(\alpha(n))+4\beta(\alpha(n))^2\left(\sum_pR^2_p\right)^{-1}\right)}\\
    &\leq \frac{36C_0n^{2\varepsilon}}{n^2\left(r^{\overline{D}_2+\varepsilon}-\mathcal{O}(n^{-2})\right)}\\
    &\leq\frac{C_6n^{2\varepsilon}}{n^2n^{-2+4\varepsilon}}=\frac{C_6}{n^{2\varepsilon}}.\tag{4.11}\label{ineqn:4.11}
\end{align*}
Hence, putting \eqref{ineqn:4.9}, \eqref{ineqn:4.10} and \eqref{ineqn:4.11} together, we can conclude that for all $n$ large enough and $r=r_n=n^{-\frac{2-4\varepsilon}{\overline{D}_2+\varepsilon}}$, there is some constant $C_7>0$ such that
\[\mu(m_n^{\gg}>8r_n)\leq\frac{Var[\mathcal{S}_n^{\gg}]}{\mathbbm{E}[\mathcal{S}_n^{\gg}]^2}\leq\frac{C_7}{n^{\varepsilon}},\]
picking a subsequence $n_k=\lceil k^{2/\varepsilon}\rceil$, the probability is summable along the subsequence which means by the Borel-Cantelli Lemma, for $\mu$-almost every $x$, for $k$ large \[-\log{m_n^{\gg}(x)}\geq-\log{4r_{n_k}}.\] 
The proof of \eqref{ineqn:4.2} is yet complete because $m_n^{\gg}$ is not a monotone sequence, but for each $n\in[n_{k},n_{k+1}]$,\[-\log{m_n^{\gg}(x)}\geq-\log\min_{\substack{0\leq i\leq n_{k}\\2n_{k+1}/3\leq j<n_{k}}}d\left(T^ix,T^jx\right)=:-\log{m_{n_k}^*(x)},\]and as for all $k$, $m_{n_{k+1}}^{\gg}(x)\leq m_{n_k}^*(x)$, repeating the same proof we have done for $m_n^{\gg}$ one can also show that $\liminf_{k\rightarrow\infty}\frac{\log{m^*_{n_k}}}{\log{n_{k}}}\geq\frac{2-4\varepsilon}{\overline{D}_2+\varepsilon}$ whence the lower bound will be passed to the entire tail of $-\log{m_{n_k}^{\gg}(x)}$ and hence $-\log{m_n(x)}$, which implies
\[\liminf_{n\rightarrow\infty}\frac{\log{m_n(x)}}{-\log{n}}\geq\frac{2}{\overline{D}_2}\]for $\mu$-almost every $x$ by sending $\varepsilon\rightarrow{0}$.
\end{proof}
\begin{remark}
    \eqref{ineqn:4.1.1} and \eqref{ineqn:4.2} still hold if decay of correlations is exponential with respect to other Banach function spaces $\mathcal{B}$, $\mathcal{B}'$, for example, both observables are in $\mathcal{BV}$ or $Lip$, where $Lip:=\{f\in C(X): f \text{ is Lipschitz}\}$, as long as $\beta(\alpha(n))\|\mathbbm{1}_{p,r}\|_{\mathcal{B}}\|\mathbbm{1}_{q,r}\|_{\mathcal{B}'}$ remains an admissible term. For example, if the system has decay of correlations for Lipschitz observables, one can replace $\mathbbm{1}_{p,r}$ functions with $\{\rho_p^r\}_{p=1}^{k(r)}$, a set of discretisation functions defined in \cite{gouezel:hal-03788538}, although it requires more machinery to adjust the proof for \eqref{ineqn:4.2} and Lemma\ref{lemma:4-mixing}. Also, instead of exponential decay of correlations, the proof remains valid under stretched exponential decay by manipulating the scale of $k$ in $\alpha(n)=(\log{n})^k$.
\end{remark}
\section*{Acknowledgements}
I am grateful for various inspirations and help, especially the proof of Lemma~\ref{lemma:RenyiEntropy}, provided by my supervisor Mike Todd, and I also like to thank several useful comments from Lingmin Liao and Jerome Rousseau.

\printbibliography

@article{haydn_vaienti_2010, title={The Rényi entropy function and the large deviation of short return times}, volume={30}, number={1}, journal={Ergodic Theory and Dynam. Systems}, publisher={Cambridge University Press}, author={Haydn, Nicolai and Vaienti, Sandro}, year={2010}, pages={159–179}}

@article{barros2019shortest,
  title={On the shortest distance between orbits and the longest common substring problem},
  author={Barros, Vanessa and Liao, Lingmin and Rousseau, J{\'e}r{\^o}me},
  journal={Adv. Math.},
  volume={344},
  pages={311--339},
  year={2019},
  publisher={Elsevier},
}

@inproceedings{barros2021shortest,
  title={Shortest distance between multiple orbits and generalized fractal dimensions},
  author={Barros, Vanessa and Rousseau, J{\'e}r{\^o}me},
  booktitle={Annales Henri Poincar{\'e}},
  volume={22},
  number={6},
  pages={1853--1885},
  year={2021},
  organization={Springer}
}

@article{ARRATIA198513,
title = {An Erdös-Rényi law with shifts},
journal = {Adv. Math.},
volume = {55},
number = {1},
pages = {13-23},
year = {1985},
doi = {https://doi.org/10.1016/0001-8708(85)90003-9},
author = {Richard Arratia and Michael S Waterman},
}

@article{10.2307/30243633,
 author = {Pierre Collet and Cristian Giardina and Frank Redig},
 journal = {Ann. Appl. Probab.},
 number = {4},
 pages = {1581--1602},
 publisher = {Institute of Mathematical Statistics},
 title = {Matching with Shift for One-Dimensional Gibbs Measures},
 volume = {19},
 year = {2009},
}

@article {MR1483874,
    AUTHOR = {Galves, A. and Schmitt, B.},
     TITLE = {Inequalities for hitting times in mixing dynamical systems},
   JOURNAL = {Random Comput. Dynam.},
    VOLUME = {5},
      YEAR = {1997},
    NUMBER = {4},
     PAGES = {337--347},
   MRCLASS = {60F05 (28D05 60G10)},
  MRNUMBER = {1483874}
}

@ARTICLE{5895075,
  author={C. Gabriela  and V. Girardin and  L. Loïck},
  journal={	IEEE Trans. Inform. Theory}, 
  title={Computation and Estimation of Generalized Entropy Rates for Denumerable Markov Chains}, 
  year={2011},
  volume={57},
  number={7},
  pages={4026-4034}
}

@article{PZ1932,
        author={Paley,Raymond. and Zygmund,Antoni},
        year={1932},
        title={On some series of functions, (3)},
        journal={Math. Proc. Cambridge Philos. Soc.},
        volume={28},
        number={2},
        pages={190-205}
}

@article{Sarig03,
author = {Sarig, Omri},
year = {2003},
month = {01},
pages = {1751-1758},
title = {Existence of Gibbs measures for countable Markov shifts},
volume = {131},
journal = {Proc. Amer. Math. Soc.}
}

@article{article,
author = {Rufus Bowen and David Ruelle},
year = {2008},
month = {01},
pages = {1-83},
title = {Equilibrium state and the Ergodic theory of Anosov diffeomorphisms},
volume = {470},
journal = {Lecture Notes in Mathematics}
}

@article{Dembo1994LimitDO,
  title={Limit Distribution of Maximal Non-Aligned Two-Sequence Segmental Score},
  author={Amir Dembo and Samuel Karlin and Ofer Zeitouni},
  journal={Ann. Probab.},
  year={1994},
  volume={22},
  pages={2022-2039}
}

@article{10.1214/aop/1176988492,
author = {Amir Dembo and Samuel Karlin and Ofer Zeitouni},
title = {{Critical Phenomena for Sequence Matching with Scoring}},
volume = {22},
journal = {Ann.Probab.},
number = {4},
publisher = {Institute of Mathematical Statistics},
pages = {1993 -- 2021},
year = {1994},
doi = {10.1214/aop/1176988492}
}

@article{sarig_1999,
title={Thermodynamic formalism for countable Markov shifts}, 
volume={19}, 
DOI={10.1017/S0143385799146820}, 
number={6}, 
journal={Ergodic Theory Dynam. Systems}, 
publisher={Cambridge University Press}, 
author={Omri M. Sarig}, 
year={1999}, pages={1565–1593}}

@unpublished{gouezel:hal-03788538,
  TITLE = {{Minimal distance between random orbits}},
  AUTHOR = {Gou{\"e}zel, S{\'e}bastien and Rousseau, J{\'e}r{\^o}me and Stadlbauer, Manuel},
  YEAR = {2022},
  MONTH = Sep,
  HAL_ID = {hal-03788538},
  HAL_VERSION = {v1},
  eprint={arXiv:2209.13240}
}

@article{HolNicTor,
author = {Holland, Mark and Nicol, Matthew and Torok, Andrei},
year = {2012},
month = {02},
pages = {},
title = {Extreme value theory for non-uniformly expanding dynamical systems},
volume = {364},
journal = {Trans. Amer. Math. Soc.},
doi = {10.1090/S0002-9947-2011-05271-2}
}

@article{cmp/1103941781,
author = {Gerhard Keller},
title = {{On the rate of convergence to equilibrium in one-dimensional systems}},
volume = {96},
journal = {Comm. Math. Phys.},
number = {2},
publisher = {Springer},
pages = {181 -- 193},
year = {1984},
doi = {cmp/1103941781}
}

@article{Aaronson2007LOCALLT,
  title={Local Limit Theorems for Gibbs-Markov Maps},
  journal={Stoch.Dyn},
  pages={193-237},
  vol={2},
  author={Jon Aaronson and Manfred Denker},
  year={2007}
}

@article{LIVERANI_SAUSSOL_VAIENTI_1999, title={A probabilistic approach to intermittency}, volume={19},  number={3}, journal={Ergodic Theory and Dynam. Systems}, 
author={Liverani, CARLANGELO and Saussol, BENOÎT and Vaienti, SANDRO}, year={1999}, pages={671–685}}
\end{document}